\documentclass[11pt]{amsart}
\usepackage[left=1.5in, right=1.5in]{geometry}
\usepackage[latin1]{inputenc}
\usepackage{amsfonts}
\usepackage[toc,page,title,titletoc,header]{appendix}
\usepackage{graphicx,psfrag,epsfig,multirow,caption,subcaption}
\usepackage{amssymb,amsmath,amscd,amsthm,amssymb,verbatim,hyperref,setspace}
\numberwithin{equation}{section}
\usepackage{mathrsfs,wrapfig}
\usepackage{indentfirst}
\usepackage{extarrows}
\newtheorem{theo}{Theorem}[section]
\newtheorem{lem}{Lemma}[section]
\newtheorem{pro}{Proposition}[section]
\newtheorem{cor}{Corollary}[section]

\title[Uniform hyperbolicity of invariant cylinder]{Uniform hyperbolicity of invariant cylinder}
\author{Chong-Qing Cheng}
\address{Department of mathematics, Nanjing Univerisity, Nanjing 210093, China}
\email{chengcq@nju.edu.cn.}
\begin{document}
\maketitle
\begin{abstract} For a nearly integrable Hamiltonian systems $H=h(p)+\epsilon P(p,q)$ with $(p,q)\in\mathbb{R}^3\times\mathbb{T}^3$, large normally hyperbolic invariant cylinders exist along the whole resonant path, except for the $\sqrt{\epsilon}^{1+d}$-neighborhood of finitely many double resonant points. It allows one to construct diffusion orbits to cross double resonance.
\end{abstract}
\begin{spacing}{0.5}
\tableofcontents
\end{spacing}
\renewcommand\contentsname{Index}

\section{Introduction and the main result}
\setcounter{equation}{0}
In this paper, we study small perturbations of integrable Hamilton systems with three degrees of freedom
\begin{equation}\label{nearlyintegrable}
H(p,q)=h(p)+\epsilon P(p,q), \qquad (p,q)\in\mathbb{R}^3\times\mathbb{T}^3,
\end{equation}
where $\partial^2h(p)$ is positive definite, both $h$ and $P$ are $C^{\kappa}$-functions with $\kappa\ge6$. In the energy level $H^{-1}(E)$ with $E>\min h$, we search for invariant cylinders along resonant path. An irreducible integer vector $k'\in\mathbb{Z}^3\backslash\{0\}$ determines a resonant path
$$
\Gamma'=\{p\in h^{-1}(E):\langle\partial h(p),k'\rangle=0\}.
$$
A point $p''\in\Gamma'$ is called double resonant if $\exists$ another irreducible vector $k''\in\mathbb{Z}^3\backslash\{0\}$, independent of $k'$, such that $\langle k'',\partial h(p'')\rangle=0$ holds as well. There are infinitely many double resonant points, but only strong double resonance causes trouble. A double resonance is called strong if $|k''|$ is not so large.

To make things simpler we introduce a symplectic coordinate transformation
$$
\mathfrak{M}:\qquad u=M^tq,\qquad v=M^{-1}p,
$$
where the matrix is made up by three integer vectors $M=(k'',k',k_3)$. As both $k'$ and $k''$ are irreducible, $\exists$ $k_3\in\mathbb{Z}^3$ such that $\mathrm{det}M=1$. There are infinitely many choices for $k_3$, we choose that $k_3$ so that $|k_3|$ is the smallest one. For simplicity of notation, we assume that the Hamiltonian of (\ref{nearlyintegrable}) is already under such transformation.

To study invariant cylinders we need normal form around a double resonance. We introduce a coordinate transformation $\Phi_{\epsilon F}$ which is defined as the time-$2\pi$-map $\Phi_{\epsilon F}=\Phi^t_{\epsilon F}|_{t=2\pi}$ of the Hamiltonian flow generated by the function $\epsilon F(p,q)$. This function solves the homological equation
\begin{equation}\label{homologicalequation}
\Big\langle\frac{\partial h}{\partial p}(p''),\frac{\partial F}{\partial q}\Big\rangle=-P(p,q)+Z(p,q)
\end{equation}
where
$$
Z(p,q)=\sum_{\ell\in\mathbb{Z}^3,\ell_3=0}P_{\ell}(p)e^{i(\ell_1q_1+\ell_2q_2)},
$$
in which $P_{\ell}$ represents the Fourier coefficient of $P$. Expanding $F$ into Fourier series and comparing both sides of the equation we obtain
$$
F(p,q)=\sum_{\ell\in\mathbb{Z}^3,\ell_3\ne 0}\frac{iP_{\ell}(p)}{\langle\ell,\partial h(p'')\rangle}e^{i\langle \ell,q\rangle}.
$$
Under the transformation $\Phi_{\epsilon F}$ we obtain a new Hamiltonian
\begin{align*}
\Phi_{\epsilon F}^*H=&h(p)+\epsilon Z(p,q)+\epsilon\Big\langle\frac{\partial h}{\partial p}(p)-\frac{\partial h}{\partial p}(p''),\frac{\partial F}{\partial q}\Big\rangle\\
&+\frac {\epsilon^2}2\int_0^1(1-t)\{\{H,F\},F\}\circ\Phi_{\epsilon F}^tdt.
\end{align*}
To solve Equation (\ref{homologicalequation}), we do not have problem of small divisor, $|\langle\ell,\partial h(p'')\rangle|=|\ell_3\omega_3|$, where $\omega_3=\partial_3h(p'')\ne 0$ since $h(p'')>\min h$.

The function $\Phi_{\epsilon F}^*H(u,v)$ determines its Hamiltonian equation
\begin{equation}\label{H-equation}
\frac{dq}{dt}=\frac{\partial}{\partial p}\Phi_{\epsilon F}^*H,\qquad \frac{dp}{dt}=-\frac{\partial}{\partial q}\Phi_{\epsilon F}^*H.
\end{equation}
For this equation we introduce another transformation (we call it homogenization)
$$
\tilde G_{\epsilon}=\frac 1{\epsilon}\Phi_{\epsilon F}^*H, \qquad\tilde y=\frac 1{\sqrt{\epsilon}}\Big(p-p''\Big), \qquad \tilde x=q, \qquad s=\sqrt{\epsilon}t,
$$
where $\tilde x=(x,x_3)$, $\tilde y=(y,y_3)$, $x=(x_1,x_2)$, $y=(y_1,y_2)$. Using the new canonical variables $(\tilde x,\tilde y)$ and the new time $s$, Equation (\ref{H-equation}) appears to be the Hamiltonian equation with its generating function as the following:
\begin{equation}\label{Hamiltonian}
\tilde G_{\epsilon}=\frac 1{\epsilon}\Big(h(p''+\sqrt{\epsilon}\tilde y)-h(p'')\Big)
-V(x)+\sqrt{\epsilon}\tilde R_{\epsilon}(\tilde x,\tilde y),
\end{equation}
where $V=-Z(p'',x)$ and
\begin{equation}\label{remainder}
\begin{aligned}
\tilde R_{\epsilon}&=\tilde R_1+\tilde R_2+\tilde R_3,\\
\tilde R_1&=\frac 1{\sqrt{\epsilon}}\Big[Z(p''+\sqrt{\epsilon}\tilde y,\tilde x)-Z(p'',\tilde x)\Big],\\
\tilde R_2&=\frac 1{\sqrt{\epsilon}}\Big\langle\frac{\partial h}{\partial p}\Big(p''+\sqrt{\epsilon}\tilde y\Big)-\frac{\partial h}{\partial p}(p''),\frac{\partial F}{\partial q}\Big\rangle,\\
\tilde R_3&=\frac {\sqrt{\epsilon}}2\int_0^1(1-t)\{\{H,F\},F\}\circ\Phi_{\epsilon F}^tdt.
\end{aligned}
\end{equation}
Let $D$ be a positive number independent of $\epsilon$, $\sigma\in(0,\frac 12)$. In the domain
$$
\Omega_{\epsilon}=\Big\{(\tilde x,\tilde y):|\tilde y|\le D\epsilon^{\sigma-\frac 12},\tilde x\in\mathbb{T}^3\Big\},
$$
the term $|\sqrt{\epsilon}\tilde R_i|_{C^{\ell-2}}$ is bounded by a small number of order $O(\epsilon^{\sigma})$ (for $i=1,2,3$). If we introduce a symplectic coordinate transformation again
$$
\mathfrak{S}:\qquad I=\frac{\omega_3}{\sqrt{\epsilon}}y_3,\qquad \theta=\frac{\sqrt{\epsilon}}{\omega_3}x_3,
$$
then $\frac{\partial\tilde G_{\epsilon}}{\partial I}=1+O(\epsilon^{\sigma})$ holds in the domain $\mathfrak{S}\Omega_{\epsilon}$. Therefore, there is a unique function $I=-G_{\epsilon}(x,y,\theta)$ which solves the equation
\begin{equation}\label{reduction}
\tilde G_{\epsilon}\Big(x,\frac{\omega_3}{\sqrt{\epsilon}}\theta,y,-\frac{\sqrt{\epsilon}}{\omega_3} G_{\epsilon}(x,y,\theta)\Big)=0,
\end{equation}
where we use the fact that $h(p'')=E$. The dynamics of $\tilde G_{\epsilon}$ restricted on the energy level set $\tilde G_{\epsilon}^{-1}(0)$ is equivalent to the dynamics of $G_{\epsilon}$ where $\theta$ plays the role of time. Let $\Phi_{G_{\epsilon}}^{\theta}$ denote the Hamiltonian flow produced by $G_{\epsilon}$.

To state our result, let us introduce some notations. A manifold with boundary is called cylinder if it is homeomorphic to the standard cylinder $\mathbb{T}\times[0,1]$. A typical case is a cylinder made up by periodic orbits of an autonomous Hamiltonian system where different orbit lies on different energy level. The cylinder is denoted by $\Pi_{E_1,E_2,g}$ if all orbits are associated with the same first homology class $g$ and they lie on the level set with energy $E_1$ to the set with $E_2$. The cylinder is invariant for the Hamiltonian flow. If the system is under small time-periodic perturbation, the time-periodic map generated by the Hamiltonian flow is a small perturbation of the original map. The cylinder may survive small perturbations of the map, denoted by $\Pi_{E_1,E_2,g}^{\epsilon}$ which is a section of a manifold homeomorphic to $\Pi_{E_1,E_2,g}\times\mathbb{T}$. We also call it cylinder.

The Tonelli Hamiltonian $G_{\epsilon}$ determines a Tonelli Lagrangian through the Legendre transformation. So, the $\alpha$- and $\beta$-function are well defined, denoted by $\alpha_{G_{\epsilon}}$ and $\beta_{G_{\epsilon}}$ respectively. They define the Legendre-Fenchel duality between the first homology and the first cohomology $\mathscr{L}_{\beta_{G_{\epsilon}}}$: a first cohomology class $c\in\mathscr{L}_{\beta_{G_{\epsilon}}}(g)$ if $\alpha_{G_{\epsilon}}(c)+\beta_{G_{\epsilon}}(g)=\langle c,g\rangle$. By adding a constant to $G_{\epsilon}$ we can assume $\min\alpha_{G_{\epsilon}}=0$. Once a Lagrangian $L$ is fixed, we also use $\mathscr{L}_{\beta_{L}}$ to denote the Legendre-Fenchel duality.

The Hamiltonian $G_{\epsilon}$ produces a map $\mathscr{L}_{G_{\epsilon}}$: $T^*\mathbb{T}^2\times\mathbb{T}\to T\mathbb{T}^2\times\mathbb{T}$: $(x,y,\theta)\to(x,\dot x,\theta)$ where $\dot x=\partial_yG_{\epsilon}(x,y,\theta)$. In this paper, a set in $T^*\mathbb{T}^2\times\mathbb{T}$ as well as its time-$2\pi$-section is called Mather set (Aubry set or Ma\~n\'e set) if its image under the map $\mathscr{L}_{G_{\epsilon}}$ is a Mather set (Aubry set or Ma\~n\'e set) in the usual definition. The following theorem is the main result of this paper:

\begin{theo}\label{mainresult}
For a class $g\in H_1(\mathbb{T}^2,\mathbb{R})$, there is an open-dense set $\mathfrak{V}\subset C^{r}(\mathbb{T}^2,\mathbb{R})$ $(r\ge 5)$. For each $V\in\mathfrak{V}$, there exists $\epsilon_0>0$ such that for each $\epsilon\in(0,\epsilon_0)$
\begin{enumerate}
  \item there exist finitely many normally hyperbolic invariant cylinders for the map $\Phi_{G_{\epsilon}}^{2\pi}$:
     $\Pi_{\epsilon^d,E_0+\delta,g}^{\epsilon}$, $\Pi_{E_0-\delta,E_1+\delta,g}^{\epsilon},\cdots,\Pi_{E_{j-1}-\delta,E_j+\delta,g}^{\epsilon}$, $\Pi_{E_j-\delta,D\epsilon^{\sigma-1/2},g}^{\epsilon}$ where the integer $j$, the numbers $E_j>\cdots>E_1>E_0>0$, the small numbers $\delta, d>0$ and the normal hyperbolicity of each cylinder are all independent of $\epsilon$;
  \item for each $c\in\mathscr{L}_{\beta_{G_{\epsilon}}}(\nu g)$
     \begin{enumerate}
     \item[$\bullet$] $\exists$ $N>1$ such that if $\alpha_{G_{\epsilon}}(c)\in (N\epsilon^d, E_0)$, the Aubry set lies on $\Pi_{\epsilon^d,E_0+\delta,g}^{\epsilon}$;
     \item[$\bullet$] if $\alpha_{G_{\epsilon}}(c)\in (E_i, E_{i+1})$, the Aubry set lies on $\Pi_{E_{i}-\delta,E_{i+1}+\delta,g}^{\epsilon}$ where the index $i$ ranges over the set $\{0,1,\cdots,j-1\}$;
     \item[$\bullet$] if $\alpha_{G_{\epsilon}}(c)\in (E_j,D\epsilon^{\sigma-\frac12})$, the Aubry set lies on $\Pi_{E_j-\delta,D\epsilon^{\sigma-1/2},g}^{\epsilon}$;
     \item[$\bullet$] if $\alpha_{G_{\epsilon}}(c)=E_i$, the Aubry set has two connected components, one is on $\Pi_{E_{i-1}-\delta,E_{i}+\delta,g}^{\epsilon}$ and another one is on $\Pi_{E_{i}-\delta,E_{i+1}+\delta,g}^{\epsilon}$.
      \end{enumerate}
\end{enumerate}
\end{theo}

Back to the original coordinates, this theorem implies the following. Along the resonant path $\Gamma'$, there are some invariant cylinders staying on the energy level $H^{-1}(E)$, which range over the extend from $\epsilon^{\sigma}$-away to $\epsilon^{(1+d)/2}$-close to the double resonant point. By a result in \cite{Lo}, the resonant path $\Gamma'$ is covered by disks $\{|p-p''_i|\le T(p''_i)\epsilon^{\sigma}\}$ where $\sigma=\frac 17$ and $T(p''_i)$ is the period of the double resonance, i.e. it is the smallest one among those integers $K\in\mathbb{Z}$ so that $K\partial h(p''_i)\in\mathbb{Z}^3$. We shall see later, for a generic potential $V$, only finitely many points $\{p''_i\}$ (independent of $\epsilon$) have to be treated as strong double resonance. It follows from Theorem \ref{mainresult} that

\begin{theo}\label{consequence}
Given a resonant path $\Gamma'\subset h^{-1}(E)$ and a generic potential $V$, there are small numbers $\epsilon_0,d>0$ so that for every $\epsilon\in(0,\epsilon_0)$, the whole path $\Gamma'$ is covered by the NHICs of the Hamiltonian flow $\Phi_H^t$, except for the $\sqrt{\epsilon}^{1+d}$-neighborhood of finitely many strong double resonant points, in the sense that the Aubry set lies in the cylinders if its rotation vector is $\sqrt{\epsilon}^{1+d}$-away from those points. The number of the strong double resonant points is independent of $\epsilon$.
\end{theo}
The result in \cite{CZ1} plays important role in this paper. It is for the minimal periodic orbit of Tonelli Lagrangian of two degrees of freedom. Let $L$ be a Tonelli Lagrangian and let $\mathfrak{M}(L)$ be the set of Borel probability measures on $T\mathbb{T}^2$, which are invariant for the Lagrange flow $\phi_L^t$ produced by $L$. Each $\mu\in\mathfrak{M}(L)$  is associated with a rotation vector $\rho(\mu)\in H_1(\mathbb{T}^2,\mathbb{R})$ s.t. for every closed 1-form $\eta$ on $\mathbb{T}^2$ one has
$$
\langle[\eta],\rho(\mu)\rangle=\int\eta d\mu.
$$
Let $\mathfrak{M}_{\omega}(L)=\{\mu\in\mathfrak{M}(L):\rho(\mu)=\omega\}$, an invariant measure $\mu$ is called minimal with the rotation vector $\omega$ if
$$
\int Ld\mu=\inf_{\nu\in\mathfrak{M}_{\omega}(L)}\int Ld\nu.
$$
A rotation vector $\omega\in H_1(\mathbb{T}^2,\mathbb{R})$ is called resonant if there exists a non-zero integer vector $k\in\mathbb{Z}^2$ such that $\langle\omega,k\rangle=0$. For two-dimensional torus, it uniquely determines an irreducible element $g\in H_1(\mathbb{T}^2,\mathbb{Z})$ and a positive number $\lambda>0$ such that $\omega=\lambda g$ if $\omega$ is resonant. Each orbit in the support of minimal measure $\mu$ is periodic if and only if $\rho(\mu)$ is resonant. Let $E=\alpha(\mathscr{L}_{\beta_{L}}(\lambda g))$, such periodic orbit is called $(E,g)$-{\it minimal}. The following result has been proved (Theorem 2.1 of \cite{CZ1}):
\begin{theo}\label{mainth}
Given a class $g\in H_1(\mathbb{T}^2,\mathbb{Z})$ and two positive numbers $E''>E'>0$, there exists an open-dense set $\mathfrak{V}\subset C^r(\mathbb{T}^2,\mathbb{R})$ with $r\ge5$ such that for each $V\in\mathfrak{V}$, it holds simultaneously for all $E\in[E',E'']$ that every $(E,g)$-minimal periodic orbit of $L+V$ is hyperbolic. Indeed, except for finitely many $E_i\in[E',E'']$, there is only one $(E,g)$-minimal orbit for $E\ne E_i$ and there are two $(E,g)$-minimal orbits for $E=E_i$. Therefore, these $(E,g)$-minimal periodic orbits make up finitely many pieces of NHICs.
\end{theo}

\section{NHIC around double resonant point}
\setcounter{equation}{0}
The main result in this section is Theorem \ref{cylinderforepsilon}. It verifies that a NHIC for the map $\Phi_{G_{\epsilon}}^{2\pi}$ extends from $\epsilon^{d}$-neighborhood of the double resonant point to a place which is of order $O(1)$-away from the double resonant point.

As the first step we need to find the explicit expression of $G_{\epsilon}$. We do not try to find the one which is valid for the whole region $\Omega_{\epsilon}$. Instead we are satisfied with getting a local expression when it is restricted on $\{|p-p'_i|\le O(\sqrt{\epsilon})\}$ where $p'_i\in\Gamma'\cap\{|p-p''|\le D\epsilon^{\sigma}\}$. Along the path $\Gamma'\cap\{|p-p''|\le\epsilon^{\sigma}\}$ we choose points $\{p'_i\}$ such that $p'_0=p''$, $\partial_1 h(p'_i)=Ki\sqrt{\epsilon}$, where $K>0$ is an integer, independent of $\epsilon$. Let $\omega_{3,i}=\partial_3h(p'_i)$, we introduce coordinate rescaling and translation
\begin{equation}\label{energylevel}
\Big(y,\frac{\sqrt{\epsilon}}{\omega_{3,i}}I\Big)=\frac{1}{\sqrt{\epsilon}}(p-p'_i),\qquad \theta=\frac{\sqrt{\epsilon}}{\omega_{3,i}}x_3.
\end{equation}
Let $Ki=\Omega_i$, we expand $\tilde G_{\epsilon}$ in $O(\sqrt{\epsilon})$ neighborhood of $p'_i$ and get a local expression
\begin{equation}\label{local}
\begin{aligned}
\tilde G_{\epsilon}=&I+\Omega_iy_1+\frac 12\Big\langle\tilde A_i\Big(y,\frac{\sqrt{\epsilon}}{\omega_{3,i}} I\Big),\Big(y,\frac{\sqrt{\epsilon}}{\omega_{3,i}}I\Big)\Big\rangle\\
&-V(x)+\sqrt{\epsilon}\tilde R_h\Big(y,\frac{\sqrt{\epsilon}}{\omega_{3,i}} I\Big)+\sqrt{\epsilon}\tilde R_{\epsilon}\Big(x,\frac{\omega_{3,i}}{\sqrt{\epsilon}}\theta, \Big(\sqrt{\epsilon}y,\frac{\epsilon}{\omega_{3,i}}I\Big)+p'_i\Big)
\end{aligned}
\end{equation}
where $\tilde A_i=\frac{\partial^2h}{\partial p^2}(p'_i)$ and term $\tilde R_h$ represents the following
$$
\frac 1{\sqrt{\epsilon}^{3}}\left[h\Big(p'_i+\Big(\sqrt{\epsilon}y,\frac{\epsilon}{\omega_{3,i}} I\Big)\Big)-\Big[h(p'_i)+I+\Omega_iy_1+\frac 12\Big\langle\tilde A_i\Big(y,\frac{\sqrt{\epsilon}}{\omega_{3,i}} I\Big),\Big(y,\frac{\sqrt{\epsilon}}{\omega_{3,i}}I\Big)\Big\rangle\Big]\right].
$$
For $|p'_i|\le O(\epsilon^{\sigma})$ and $|y|\le O(1)$ and $|I|\le O(\frac{\omega_3}{\sqrt{\epsilon}})$ both $\sqrt{\epsilon}\tilde R_h$ and $\sqrt{\epsilon}\tilde R_{\epsilon}$ are bounded by a quantity of order $O(\epsilon^{\sigma})$ in $C^{\ell-2}$-topology. From the expression of $\tilde G_{\epsilon}$ in (\ref{local}) we get a local solution of the equation (cf. Equation \ref{reduction})
\begin{equation}\label{reduc}
\tilde G_{\epsilon}\Big(x,\frac{\omega_{3,i}}{\sqrt{\epsilon}}\theta,y,-\frac{\sqrt{\epsilon}}{\omega_{3,i}} G_{\epsilon,i}(x,y,\theta)\Big)=0,
\end{equation}
which takes the form of $G_{\epsilon,i}(x,y,\theta)=G_i(x,y)+\epsilon^{\sigma}R_{\epsilon,i}(x,y,\theta)$ where
\begin{equation}\label{2dHmil}
G_{i}(x,y)=\Omega_iy_1+\frac 12\langle Ay,y\rangle-V(x).
\end{equation}
where $A$ is obtained from $\tilde A_0$ by eliminating the third row and the third column. At first view, the matrix $A$ should come from $\tilde A_i$ in the same way. However, we use the property
$|p'_i-p''|\le O(\epsilon^{\sigma})$ so that we can put the difference term into the remainder.

In this section, we study the case when $i=0$, i.e. the system is restricted in $K\sqrt{\epsilon}$-neighborhood of the double resonant point. If we ignore the small term $\epsilon^{\sigma}R_{\epsilon,0}$ in the equation (\ref{2dHmil}), the truncated system $G_0$ has two degrees of freedom only. Let $L_{G_0}$ be the Lagrangian determined by the truncated system, we get periodic orbit with rotation vector $\lambda g$ by searching for the minimizer $\gamma(\cdot,E,g,x)$ of the Lagrange action
\begin{equation}\label{functionofaction}
F(x,E,g)=\inf_{\stackrel {\gamma(0)=\gamma(\frac {2\pi}{\lambda})=x}{\scriptscriptstyle [\gamma]=g}}\int_{0}^{\frac {2\pi}{\nu}}L_{G_0}(\gamma(t),\dot\gamma(t))dt,
\end{equation}
where $g\in H_1(\mathbb{T}^2,\mathbb{Z})$ is irreducible, $E=\alpha(\mathscr{L}_{\beta_L}(\lambda g))$. If $F(\cdot,E,g)$ reaches its minimum at $x^*$, then $d\gamma(\cdot,E,g,x^*)$ is the minimal periodic orbit we are looking for \cite{CZ1}.

The minimizer $\gamma(\cdot,E,g,x^*)$ produces a periodic orbit $z_{E,g}(t)=(x_{E,g}(t),y_{E,g}(t))$ of $\Phi^{t}_{G_0}$, where $x_{E,g}(t)=\gamma(t,E,g,x^*)$. As the Hamiltonian is autonomous, the orbit $z_{E,g}(t)$ stays in certain energy level $G^{-1}_0(E)$. When the energy $E$ decreases, $\lambda$ also decrease. We assume $\min V=0$, then there are two possibilities:

(1), $\lambda\downarrow\lambda_0>0$ as $E\downarrow 0$. In this case, certain periodic orbit $z^*(t)\subset G^{-1}_0(0)$ such that $z_{E,g}(t)\to z^*(t)$. It is possible, we have an example. Let
$$
L=\frac 12\dot x_1^2+\frac {1}2\dot x_2^2+V(x)
$$
where $V$ satisfies the conditions: $x=0$ is the minimal point of $V$ only; there exist two numbers $d>d'>0$ such that for any closed curve $\gamma$: $[0,1]\to\mathbb{T}^2$ passing through the origin with $[\gamma]\ne 0$ one has
$$
\int_0^1V(\gamma(s))ds\ge d;
$$
$V=d'+(x_2-a)^2$ when it is restricted a neighborhood of circle $x_2=a$. For $g=(1,0)$, $\lambda\downarrow \sqrt{2d'}$ as $E\downarrow0$, $z^*(t)=(\frac t{\sqrt{2d'}},a)$. No problem of double resonance in this case.

(2), $\lambda\downarrow 0$ as $E\downarrow 0$. It is typical that $V$ attains its minimum at one point which correspond a fixed point of the Hamiltonian flow $\Phi_{G_0}^t$. As the period $2\lambda^{-1}\pi$ approaches infinity, the orbit $z_{E,g}(t)$ approaches homoclinic orbit(s) as $E$ approaches zero. It is possible that there are two irreducible classes $g_1,g_2\in H_1(\mathbb{T}^2,\mathbb{Z})$ and two non-negative integers $k_1,k_2$ such that $g=k_1g_1+k_2g_2$. It is a difficult part of the problem of double resonance, $\{z_{E,g}\}_{E>0}$ makes up a cylinder which takes homoclinic orbits as it boundary, which can not survive any small perturbation. In this section, we are going to study how close some invariant cylinder of $\Phi^{2\pi}_{G_{\epsilon,0}}$ can extend to.

\subsection{Hyperbolicity of minimal periodic orbit around double resonant point} At the double resonant point we have $p'=p''$ and $\omega'_1=0$. In this case,
\begin{equation}\label{Hamilton2}
G_{\epsilon,0}(x,y,\theta)=\frac 12\langle Ay,y\rangle-V(x)+\epsilon^{\sigma}R_{\epsilon,0}(x,y,\theta).
\end{equation}
The Lagrangian determined by $G_0$ takes the form
$$
L_{G_0}=\frac 12\langle A^{-1}\dot x,\dot x\rangle+V(x).
$$
For Hamiltonian system $G_0$, the minimal point of $V$ determines a stationary solution which corresponds to a minimal measure of the Lagrangian $L_{G_0}$. Up to a translation of coordinates $x\to x_0$, it is open-dense condition that

({\bf H1}): {\it $V$ attains its maximum at $x=0$ only, the Hessian matrix of $V$ at $x=0$ is positive definite. All eigenvalues of the matrix
$$
\left (\begin{matrix}0 & A\\
\partial^2_xV & 0
\end{matrix}\right )
$$
are different: $-\lambda_2<-\lambda_1<0<\lambda_1<\lambda_2$}.

If we denote by $\Lambda^+_i=(\Lambda_{xi},\Lambda_{yi})$ the eigenvector corresponding to the eigenvalue $\lambda_i$, where $\Lambda_{xi}$ and $\Lambda_{yi}$ are for the $x$- and $y$-coordinate respectively, then the eigenvector for $-\lambda_i$ will be  $\Lambda^-_i=(\Lambda_{xi},-\Lambda_{yi})$.

By the assumption ({\bf H1}), the fixed point $z=(x,y)=0$ has its stable manifold $W^+$ and its unstable manifold $W^-$. They intersect each other along homoclinic orbit. As each homoclinic orbit entirely stays in the stable as well as in the unstable manifolds, the intersection can not be transversal in the standard definition, but in the sense that
$$
T_xW^-\oplus T_xW^+=T_xG_0^{-1}(0)
$$
holds for all $x$ along minimal homoclinic curve. Without danger of confusion, we also call the intersection transversal.

Treat it as a closed curve, a homoclinic orbit $(\gamma(t),\dot\gamma(t))$ is associated with a homological class $[\gamma]=g\in H_1(\mathbb{T}^2,\mathbb{Z})$. A homoclinic orbit $(\gamma,\dot\gamma)$ is called {\it minimal} if $$
\int_{-\infty}^{\infty}L_{G_0}(\gamma(t),\dot\gamma(t))dt= \inf_{[\zeta]=[\gamma]}\int_{-\infty}^{\infty} L_{G_0} (\zeta(t),\dot\zeta(t))dt.
$$
Although there are infinitely many homoclinic orbits \cite{Z2,CC}, it is generic that there is only one minimal homoclinic orbit for each class in $g\in H_1(\mathbb{T}^2,\mathbb{Z})$. As there are countably many homological classes only, the following hypothesis is also generic:

({\bf H2}): {\it The stable manifold intersects the unstable manifold transversally along each minimal homoclinic orbit. These minimal homoclinic orbits approach the fixed point along the direction $\Lambda_1$: $\dot\gamma(t)/\|\dot\gamma(t)\| \to\Lambda_{x1}$ as $t\to\pm\infty$.}

Once fixing the homological class $g$, we denote the periodic curve $x_{E}(t)=x_{E,g}(t)$ which determines a periodic orbit $z_{E}=(x_{E},y_{E})$ in the phase space. As it stays in the energy level set $ G^{-1}_0(E)$, let us show how the period $T_{E}$ is related to the energy $E$.
\begin{lem}
Assume the hypothesis $(${\bf H2}$)$. For $g=k_1g_1+k_2g_2$ and suitably small $E>0$, the period $T_{E}$ of the orbit $z_{E}$ is related to the energy $E$ through the formula
\begin{equation}\label{period}
T_{E}=T(E,g)=\tau_{E,g}-\frac 1{\lambda_1}(k_1+k_2)\ln E
\end{equation}
where $\tau_{E,g}$ is bounded as $E\downarrow 0$.
\end{lem}
\begin{proof}
By the condition, there are two homoclinic orbits $d\gamma_1(t),d\gamma_2(t)$ such that $[\gamma_1]=g_1$, $[\gamma_2]=g_2$ and the periodic orbit $z_{E}(t)$ approaches them as $E\downarrow 0$. By the hypotheses $(${\bf H1,H2}$)$, these minimal homoclinic orbits approach the fixed point $z=0$ along the direction $\Lambda_1^{\pm}$.

Let $B_{\delta}$ be a sphere centered at $z=0$ with small radius $\delta>0$. Since $z_{E}$ approaches the homoclinic orbits, it passes through the ball if $E>0$ is small. Denote by $t^+_{E,i}$ the time when $z_{E}$ enters the ball, $t^-_{E,i}$ the subsequent time when $z_{E}$ leaves the ball, namely, $z_{E}(t)\in B_{\delta}$ for $t\in [t^+_{E,i},t^-_{E,i}]$ and $z_{E}(t)\notin B_{\delta}$ for $t\in(t^-_{E,i},t^+_{E,i+1})$. Since $g=k_1g_1+k_2g_2$, we have
$$
t^+_{E,1}<t^-_{E,1}<t^+_{E,2}<\cdots<t^+_{E,k_1+k_2}<t^-_{E,k_1+k_2}<t^+_{E,k_1+k_2+1}= t^+_{E,1}+T_{E}.
$$
We use the notation $z_{E}(t)=(x_{E}(t),y_{E}(t))$. As the minimal homoclinic orbits approach the fixed point $z=0$ along the direction $\Lambda_1^{\pm}$.
\begin{equation}\label{regularenergyeq1}
\Big\|\frac{\dot x_{E}(t_{E,i}^{\pm})}{\|\dot x_{E}(t_{E,i}^{\pm})\|}-\frac{\Lambda_{x1}}{\|\Lambda_{x1}\|} \Big\|<\frac 14
\end{equation}
holds if $E>0,\delta>0$ are suitably small.

In a suitably small neighborhood of $z=0$, we use a Birkhoff normal form
$$
G_0=\frac 12(y_1^2-\lambda_1^2x_1^2)+\frac 12(y_2^2-\lambda_2^2x_2^2)+P_3(x,y)
$$
where $P_3(x,y)=O(|(x,y)|^3)$. In such coordinates, the eigenvector for the eigenvalue $\pm\lambda_1$ is $\Lambda_1^{\pm}=(1,0,\pm\lambda_1,0)$ and that for the eigenvalue $\pm\lambda_2$ reads $\Lambda_2^{\pm}=(0,1,0,\pm\lambda_2)$.

By the method of variation of constants, we obtain the solution of the Hamilton equation generated by $G_0$
\begin{equation}\label{flateq5}
\begin{aligned}
x_i(t)=&e^{-\lambda_it}(b_{i}^{-}+F_i^-)+e^{\lambda_it}(b_{i}^{+}+F_i^+), \\
y_i(t)=&-\lambda_ie^{-\lambda_it}(b_{i}^{-}+F_i^-)+\lambda_ie^{\lambda_it}(b_{i}^{+}+F_i^+),
\end{aligned}
\end{equation}
where $b_i^{\pm}$ are constants determined by boundary condition and
\begin{align*}
F_i^-=&\frac{1}{2\lambda_i} \int_0^te^{\lambda_is}(\lambda_i\partial_{y_i}P_3+\partial_{x_i}P_3)(x(s),y(s))ds, \\
F_i^+=&\frac{1}{2\lambda_i} \int_0^te^{-\lambda_is}(\lambda_i\partial_{y_i}P_3-\partial_{x_i}P_3) (x(s),y(s))ds.
\end{align*}
Substituting $(x,y)$ with the formula (\ref{flateq5}) into $G_0$ we obtain a constraint condition for the constants $b_i^{\pm}$:
\begin{equation}\label{flateq6}
G_0(x(t),y(t))=-2(\lambda_1^2b_1^-b_1^++\lambda_2^2b_2^-b_2^+)+P_3((b^+_i+b^-_i),\lambda_i(b^+_i-b^-_i)).
\end{equation}
If $(x(\pm T),y(\pm T))\in\partial B_{\delta}$, we obtain from the theorem of Grobman-Hartman that
\begin{equation}\label{flateq7.1}
\begin{aligned}
x_i(-T)=&b_i^-e^{\lambda_iT}+b_i^+e^{-\lambda_iT}+o(\delta),\\
x_i(T)=&b_i^-e^{-\lambda_iT}+b_i^+e^{\lambda_iT}+o(\delta).
\end{aligned}
\end{equation}
As Formula (\ref{regularenergyeq1}) holds for $(x_{E},y_{E})$, the first component of $x_{E}(t)=(x_{E,1}(t),x_{E,2}(t))$ satisfies
\begin{equation}\label{flateq7}
|x_{E,1}(t^{\pm}_{E,i})|\ge\frac{\delta}{2\sqrt{1+\lambda_1^2}}, \qquad i=1,\cdots,k_1+k_2.
\end{equation}
Let $2T=t^-_{E,i}-t^+_{E,i}$. The time translation, $t^+_{E,i}\to -T$ induces $t^-_{E,i}\to T$. For sufficiently large $T>0$, it deduces from the equation (\ref{flateq7.1}) and the assumption (\ref{flateq7}) that
$$
\frac{\delta}{3\sqrt{1+\lambda_1^2}} e^{-\lambda_1T}\le |b_1^{\pm}|\le 2e^{-\lambda_1T},\qquad |b_2^{\pm}|\le 2\delta e^{-\lambda_2T},
$$
and
$$
b_1^-b_1^+<0,\qquad |P_3((b^+_j+b^-_j),\lambda_j(b^+_j-b^-_j))|\le Ce^{-3\lambda_1T},
$$
where the constant $C$ depends only on the function $P_3$. So, for suitably small ${\delta}>0$  and sufficiently large $|t_{E,i}^--t_{E,i}^+|$,  we obtain from (\ref{flateq6}) that
\begin{align*}
E\ge&\frac{2\lambda_1^2{\delta}^2}{9(1+\lambda_1^2)} e^{-\lambda_1|t_{E,i}^--t_{E,i}^+|}-8 \lambda_2^2{\delta}^2e^{-\lambda_2|t_{E,i}^--t_{E,i}^+|}-Ce^{-3\lambda_1|t_{E,i}^--t_{E,i}^+|/2}\notag\\
\ge&\frac{\lambda_1^2{\delta}^2}{9(1+\lambda_1^2)}e^{-\lambda_1|t_{E,i}^--t_{E,i}^+|}
\end{align*}
The quantity $|t_{E,i}^--t_{E,i}^+|$ becomes sufficiently large if $E>0$ is sufficiently small.
On the other hand, $E$ is obviously upper bounded by
\begin{align*}
E
\le&8\lambda_1^2{\delta}^2e^{-\lambda_1|t_{E,i}^--t_{E,i}^+|}+8 \lambda_2^2{\delta}^2e^{-\lambda_2|t_{E,i}^--t_{E,i}^+|}+Ce^{-3\lambda_1|t_{E,i}^--t_{E,i}^+|/2}\notag\\
\le&9\lambda_1^2{\delta}^2e^{-\lambda_1|t_{E,i}^--t_{E,i}^+|}.
\end{align*}
Therefore, we find the dependence of speed on the energy
\begin{equation}\label{regularenergyeq2}
|t_{E,i}^--t_{E,i}^+|=\frac 1{\lambda_1}|\ln E|-\frac 2{\lambda_1}|\ln \delta|+\tau_{E,i}
\end{equation}
where $\tau_{E,i}$ is uniformly bounded for each $i\le k_1+k_2$:
$$
\frac 1{\lambda_1}\Big(2\ln\lambda_1+\ln \frac{1}{9(1+\lambda_1^2)}\Big)\le\tau_{E,i}\le \frac 1{\lambda_1}(2\ln\lambda_1+3\ln 3).
$$
For $t\in (t^-_{E,i},t^{+}_{E,i+1})$, the point $z_{E}(t)$ does not fall into the ball $B_{\delta}$. So, the quantity $t^{+}_{E,i+1}-t^-_{E,i}$ is uniformly bounded as $E\downarrow 0$. Set
$$
\tau_{E,g}=\sum_{i=1}^{k_1+k_2}\tau_{E,i}+(t^{+}_{E,i+1}-t^-_{E,i}),
$$
we obtain the formula (\ref{period}).
\end{proof}

We are going to study the hyperbolicity of the periodic orbit $z_{E}(t)$. As the Hamilton flow $\Phi_{G_0}^t$ preserves the energy, we take a two-dimensional section $\Sigma_E\subset G_0^{-1}(E)$, which is transversal to the periodic orbit $z_{E}(t)$ at $z_{E,0}$ in the sense that
$$
T_{z_{E,0}}G_0^{-1}(E)=\mathrm{span}\{J\nabla G_0(z_{E,0}),T\Sigma_E\}.
$$
The Hamiltonian flow produces a Poincar\'e map, for which $z_{E,0}$ is periodic point (the orbit may intersect the section at several points). We study the hyperbolicity of periodic point for the Poincar\'e map. If a periodic orbit is associated with the homological class $g$ and it stays in the energy level of $E$, we call it $(E,g)$-periodic orbit.

\begin{lem}\label{cylinderlem1}
We assume the hypotheses \text{\rm ({\bf H1})} and \text{\rm ({\bf H2})}. Given a class $g\in H_1(\mathbb{T}^2,\mathbb{Z})$, we assume that, as $E \downarrow 0$, there is $(E,g)$-periodic orbit $z_{E}(t)$ which approaches two minimal homoclinic orbits $d\gamma_1$ and $d\gamma_2$ so that $g=k_1[\gamma_1]+k_2[\gamma_2]$. Then, there exists small $E'>0$ such that for each $E\in (0,E']$, there exists a two-dimensional disk $\Sigma_E\subset G_0^{-1}(E)$ which intersects the orbit $z_{E}(t)$ transversally. Restricted on the section, the Hamiltonian flow $\Phi_{G_0}^t$ induces a Poincar\'e return map $\Phi_{E}$: $\Sigma_E\to \Sigma_E$, and there exists some $\lambda>1,C>1$ independent of $E\le E'$ such that
$$
\|D\Phi_{E}(z_{E,0})v^-\|\ge C E^{-\lambda}\|v^-\|,\qquad \forall\ v^-\in T_{z_{E,0}}W^-_E;
$$
$$
\|D\Phi_{E}(z_{E,0})v^+\|\le C^{-1}E^{\lambda}\|v^-\|,\qquad \forall\ v^+\in T_{z_{E,0}}W^+_E,
$$
where $z_{E,0}$ is the point where the periodic orbit intersects $\Sigma_{E}$, $W^{\pm}_E$ denotes the stable $($unstable$)$ manifold of the periodic orbit.
\end{lem}
\begin{proof}
To study the dynamics around the minimal homoclinic orbits $d\gamma_1(t),d\gamma_2(t)$, we use new canonical coordinates $(x,y)$ such that, restricted in a small neighborhood of $z=0$, one has the form
$$
G_0=\frac 12(y_1^2-\lambda_1^2x_1^2)+\frac 12(y_2^2-\lambda_2^2x_2^2)+P_3(x,y)
$$
where $P_3(x,y)=O(\|x,y\|^3)$. In such coordinates, we use $z_{\ell}=(x_{\ell},y_{\ell})$ to denote the homoclinic orbit $d\gamma_{\ell}$ ($\ell=1,2$). We can assume $x_{\ell,1}(t)\downarrow 0$ as $t\to -\infty$, $x_{\ell,1}(t)\uparrow 0$ as $t\to \infty$ and $\dot x_{\ell}(t)/ \|\dot x_{\ell}(t)\|\to(1,0)$ as $t\to\pm\infty$. Here the notation is taken as granted: $x_{\ell}=(x_{\ell,1},x_{\ell,2})$. We choose 2-dimensional disk lying in $G_0^{-1}(E)$
$$
\Sigma^{\mp}_{E,\delta}=\{(x,y)\in\mathbb{R}^4:\|(x,y)\|\le d,G_0(x,y)=E, x_1=\pm\delta\}.
$$
Because of the special form of $G_0$, one has
$$
\Sigma^{\mp}_{0,\delta}=\{x_1=\pm\delta, y_1^2+y_2^2-\lambda_2^2x_2^2 =\lambda_1^2\delta^2 -2P_3(\pm\delta,x_2,y),\|(x,y)\|\le d\}.
$$
Let $W^-$ ($W^+$) denote the unstable (stable) manifold of the fixed point which entirely stays in the energy level set $G_0^{-1}(0)$. If $P_3=0$, the tangent vector of $W^-\cap\Sigma^-_{0,\delta}$ has the form $(0,\pm 1,0,\pm\lambda_2)$. So, the tangent vector of $W^-\cap\Sigma^-_{0,\delta}$ takes the form
\begin{equation*}\label{tangentvector}
v_{\delta}^-=(v_{x_1},v_{x_2},v_{y_1},v_{y_2})=(0,\pm 1,y_{1,\delta},\pm\lambda_2+y_{2,\delta})\in T_{z^-_{\delta}}(W^-\cap\Sigma^-_{0,\delta})
\end{equation*}
where both $y_{1,\delta}$ and $y_{2,\delta}$ are small.

Let $T^{\pm}_{\delta,\ell}$ be the time when the homoclinic orbit $z_\ell(t)$ passes through $\Sigma^{\pm}_{0,\delta}$. Since $\partial_{y_1}G_0>0$ holds at the point $z_{\ell}(t)\cap\{x_1=\pm\delta\}$, both homoclinic orbits $z_1(t)$ and $z_2(t)$ approach the fixed point in the same direction, the section $\Sigma^{\pm}_{0,\delta}$ intersects these two homoclinic orbits transversally. Let $z^{\pm}_{\delta,\ell}$ denote the intersection point. In a small neighborhood of the point $B_{\varepsilon}(z^-_{\delta,\ell})$, one obtains a map $\Psi_{0,\delta}$: $\Sigma^-_{0,\delta}\cap B_{\varepsilon}(z^-_{\delta,\ell})\to \Sigma^+_{0,\delta}$  in following way, starting from a point $z$ in this neighborhood, there is a unique orbit which moves along $z_{\ell}(t)$ and comes to a point $\Psi_{0,\delta}(z)\in \Sigma^+_{0,\delta}$ after a time approximately equal to $T_{\delta,\ell}^+-T_{\delta,\ell}^-$.

Let us fix small $D>0$. There exists $C_0>1$ (depending on $D$) such that
$$
C_0^{-1}\le \|D\Psi_{0,D}(z^-_{D,\ell})|_{T(W^-\cap\Sigma^-_{0,D})}\|, \|D\Psi_{0,D}^{-1}(z^+_{D,\ell})|_{T(W^+\cap\Sigma^+_{0,D})}\|\le C_0
$$
holds for both $\ell=i$ and $\ell=i+1$. Clearly, one has $C_0\to\infty$ as $D\to 0$.

As the homoclinic curves approach to the origin in the direction of $(1,0)$ in $x$-space, for small $\delta\ll D$, there exists a constant $\mu_1>0$ such that $\mu_1\downarrow 0$ as $D\to 0$ and
\begin{equation}\label{time}
\frac {1}{\lambda_1+\mu_1}\ln\Big(\frac{D}{\delta}\Big)\le T^{-}_{D,\ell}-T^{-}_{\delta,\ell}, T^+_{\delta, \ell}-T^+_{D,\ell}\le \frac {1}{\lambda_1-\mu_1}\ln\Big(\frac{D}{\delta}\Big).
\end{equation}
The Hamiltonian flow $\Phi^t_{G_0}$ defines a map $\Psi^-_{0,\delta,D}$: $\Sigma^-_{0,\delta}\to\Sigma^-_{0,D}$ and a map $\Psi^+_{0,\delta,D}$: $\Sigma^+_{0,D}\to\Sigma^+_{0,\delta}$: emanating from a point in $\Sigma^-_{0,\delta}$ ($\Sigma^+_{0,D}$) there exists a unique orbit which arrives $\Sigma^-_{0,D}$ ($\Sigma^+_{0,\delta}$) after a time bounded by the last formula.

Restricted in the ball $B_D$, let us consider the variational equation of the flow $\Psi_{G_0}^s$ along the homoclinic orbit $z_j(t)$. It follows from the normal form of the homogenized Hamiltonian $G_0$ that the tangent vector $(\Delta x,\Delta y)=(\Delta x_1,\Delta x_2,\Delta y_1,\Delta y_2)$ satisfies the variational equation
\begin{equation}\label{cylindereq4}
\begin{aligned}
\Delta\dot  x_i&=\Delta y_i+\sum_{j=1}^2\Big(\frac{\partial^2P}{\partial x_j\partial y_i}\Delta x_j+ \frac{\partial^2P}{\partial y_j\partial y_i}\Delta y_j\Big), \\
\Delta\dot  y_i&=\lambda_i^2\Delta x_i-\sum_{j=1}^2\Big(\frac{\partial^2P}{\partial x_j\partial x_i}\Delta x_j+ \frac{\partial^2P}{\partial y_j\partial x_i}\Delta y_j\Big), \ \ \ i=1,2.
\end{aligned}
\end{equation}
Clearly, $|\partial^2 P(z_{\ell}(t))|\le C_1\|z_{\ell}(t)\|$ with $C_1>0$ if $\|z_{\ell}(t)\|$ is small. Along the homoclinic orbit $\|z_{\ell}(t)\|$ is of the same order as $\|x_{\ell}(t)\|$ if $\|x_{\ell}(t)\|$ small. Since the homoclinic orbit approaches to the fixed point in the direction of $(\dot x,\dot y)=(1,0,\lambda_1^2,0)$, one has
$$
De^{-(\lambda_1+\mu_1)(t-T^+_{D,\ell})}\le\|x(t)|_{[T^+_{D,\ell},\infty)}\|\le De^{-(\lambda_1-\mu_1)(t-T^+_{D,\ell})}.
$$
For the initial value $\Delta z(T^+_{D,\ell})=(\Delta x(T^+_{D,\ell}),\Delta y(T^+_{D,\ell}))$ satisfying the condition
$$
|\langle\Delta z(T^+_{D,\ell}),v^-_{\delta}\rangle|\ge 2/3\|\Delta z(T^+_{D,\ell})\|\|v^-_{\delta}\|
$$
$(v^-_{\delta}=(0,\pm 1,y_{1,\delta},\pm\lambda_2+y_{2,\delta}))$ one obtains from the hyperbolicity that
$$
C_2^{-1}\|\Delta z(T^+_{D,\ell})\|e^{(\lambda_2-\mu_1)(T^+_{\delta,\ell}-T^+_{D,\ell})}\le\|\Delta z(T^+_{\delta,\ell})\|\le C_2\|\Delta z(T^+_{D,\ell})\|e^{(\lambda_2+\mu_1)(T^+_{\delta,\ell}-T^+_{D,\ell})}
$$
holds for some constant $C_2>1$ depending on $\lambda_i$ as well as on $P$. Thus, for each vector $v\in T_{z_D^+}\Sigma^+_{0,D}$ nearly parallel to $T_{z_D^+}(W^-\cap\Sigma^+_{0,D})$ in the sense that $|\langle v,v'\rangle|\ge\frac 23\|v\|\|v'\|$ holds for $v'\in T_{z_D^+}(W^-\cap\Sigma^+_{0,D})$ we obtain from the last two formulae and (\ref{time}) that
$$
C_2^{-1}\Big(\frac D{\delta}\Big)^{\frac{\lambda_2}{\lambda_1}-\mu_2}\le
\lim_{\|v\|\to 0}\frac{\|D\Psi^+_{0,\delta,D}(z^+_{D,\ell})v\|}{\|v\|}\le C_2\Big(\frac D{\delta}\Big)^{\frac{\lambda_2}{\lambda_1}+\mu_2}.
$$
Similarly, one has
$$
C_3^{-1}\Big(\frac D{\delta}\Big)^{\frac{\lambda_2}{\lambda_1}-\mu_2}\le \|D\Psi^-_{0,\delta,D}(z^-_{\delta,\ell})|_{T_{z^-_{\delta}}(W^-\cap\Sigma^-_{0,\delta})}\|\le C_3\Big(\frac D{\delta}\Big)^{\frac{\lambda_2}{\lambda_1}+\mu_2},
$$
where $C_3>1$ also depends on $\lambda_i$ as well as on $P$, $\mu_2>0$ and $\mu_2\to 0$ as $D\to 0$.

By the construction, the 2-dimensional disk $\Sigma^{-}_{0,\delta}$ intersects the unstable manifold $W^-$ along a curve. Let $\Gamma^{-}_{\delta,\ell}\subset W^{-}\cap\Sigma^{-}_{0,\delta}$ be a very short segment of the curve, passing through the point $z^{-}_{\delta,\ell}$. Pick up a point $z^*_\ell$ on the homoclinic orbit $z_\ell$ far away from the fixed point and take a 2-dimensional disk $\Sigma^*_\ell\subset G_0^{-1}(0)$ containing the point $z^*_\ell$ and transversal to the flow $\Phi^t_{G_0}$ in the sense that $T_{z^*_\ell}G_0^{-1}(0)=\text{\rm span}(T_{z^*_\ell} \Sigma_\ell,J\nabla G_0(z^*_\ell))$. The Hamiltonian flow $\Phi^t_{G_0}$ maps a point of $\Gamma^{-}_{\delta,\ell}$ to this disk provided it is close to $z^-_\ell$. In this way, one obtains a map $\Psi^{-,*}_{\delta,\ell}$: $\Sigma^-_{0,\delta}\to\Sigma^*_\ell$. Let $\Gamma^{-,*}_{\delta,\ell}= \Psi^{-,*} _{\delta,\ell} \Gamma^{-}_{\delta,\ell}$. According to the assumption ({\bf H2}), one has $T_{z^*_{\ell}}G_0^{-1}(0)=\text{\rm span}(T_{z^*_{\ell}}W^+,T_{z^*_{\delta,\ell}}W^-)$.  Thus, one also has $T_{z^*_{\ell}}G_0^{-1}(0)=\text{\rm span}(T_{z^*_{\ell}}W^+,T_{z^*_{\ell}} \Gamma^{-,*}_{\ell})$. It follows from the $\lambda$-lemma that $\Psi_{0,\delta}(\Gamma^-_{\delta,\ell})$ keeps $C^1$-close to $W^-\cap\Sigma^+_{0,\delta}$ at the point $z^+_{\delta,\ell}$ and $\Psi^{-1}_{0,\delta}(\Gamma^+_{\delta,\ell})$ keeps $C^1$-close to $W^+\cap\Sigma^-_{0,\delta}$ at the point $z^-_{\delta,\ell}$ provided $\delta>0$ is sufficiently small. As $\Psi_{0,\delta}=\Psi^-_{0,\delta,D}\circ\Psi_{0,D}\circ\Psi^+_{0,\delta,D}$, one obtains
$$
C_4^{-1}\left(\frac D{\delta}\right)^{2(\frac {\lambda_2}{\lambda_1}-\mu_2)}\le \|D\Psi_{0,\delta}(z^-_{\delta})|_{T_{z^-_{\delta}}(W^-\cap\Sigma^-_{0,\delta})}\|
\le C_4\left(\frac D{\delta}\right)^{2(\frac {\lambda_2}{\lambda_1}+\mu_2)},
$$
and
$$
C_4^{-1}\left(\frac D{\delta}\right)^{2(\frac {\lambda_2}{\lambda_1}-\mu_2)}\le \|D\Psi_{0,\delta}^{-1} (z^+_{\delta}) |_{T_{z^+_{\delta}}(W^+\cap\Sigma^+_{0,\delta})}\|
\le C_4\left(\frac D{\delta}\right) ^{2(\frac {\lambda_2}{\lambda_1}+\mu_2)},
$$
where $C_4=C_0C_2C_3>1$. See the figure below.

\begin{figure}[htp]
  \centering
  \includegraphics[width=6.5cm,height=6.7cm]{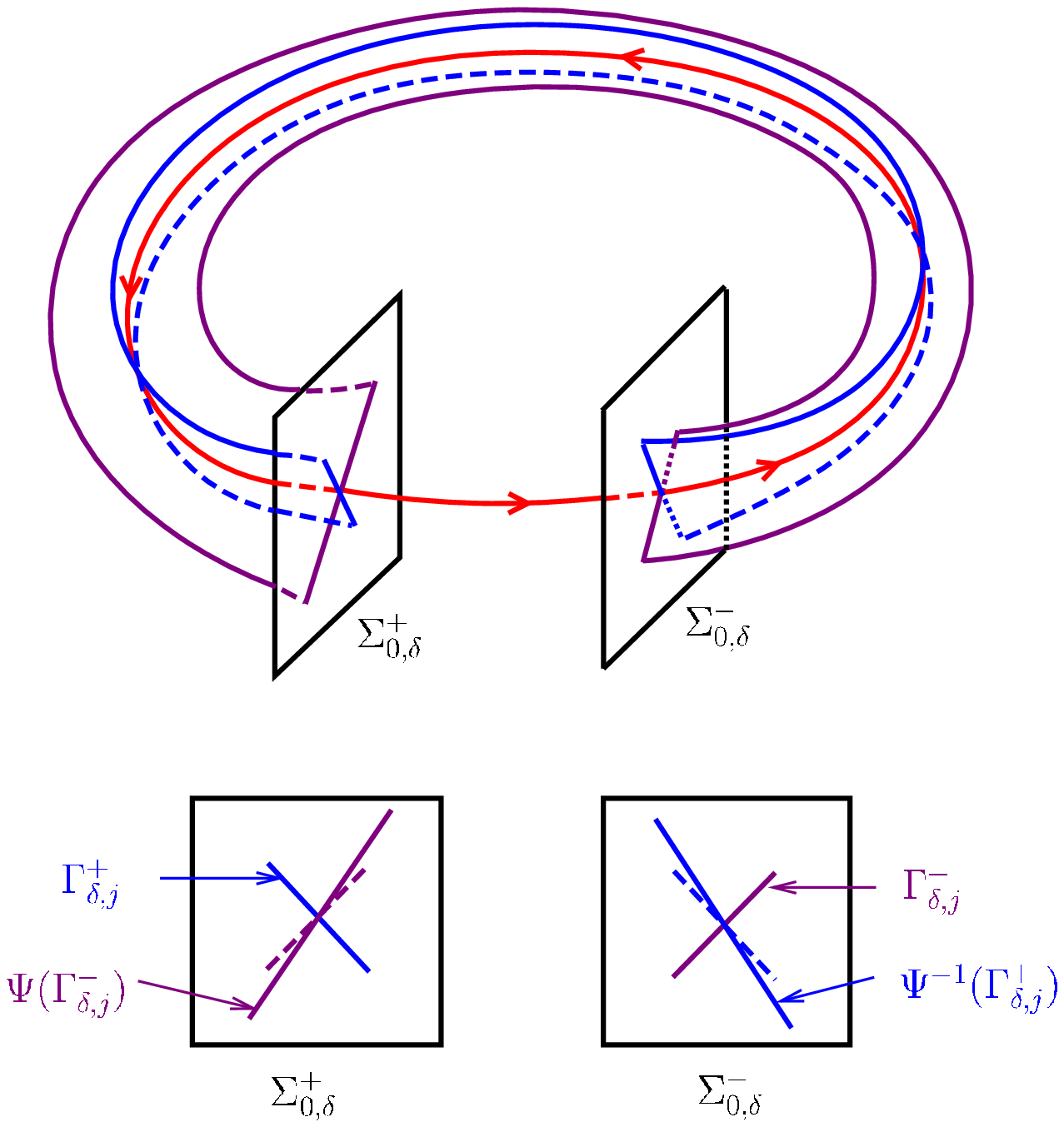}
  \label{fig7}
\end{figure}

Recall the definition, $\Sigma^{\pm}_{E,\delta}$ is a two-dimensional disk lying in the energy level set $G_0^{-1}(E)$. For $E>0$ sufficiently small, $\Sigma^{\pm}_{E,\delta}$ is $C^{r-1}$-close to $\Sigma^{\pm}_{0,\delta}$ respectively. Let $z_{E}(t)=(x_{E}(t),y_{E}(t))$ be the minimal periodic orbit staying in the energy level set $G_0^{-1}(E)$, it approaches to the homoclinic orbit as $E$ decreases to zero. Thus, for sufficiently small $E>0$, it passes through the section $\Sigma^{-}_{E,\delta}$ as well as $\Sigma^{+}_{E,\delta}$ $k_1+k_2$ times for one period. We number these points as $z^{\pm}_{E,k}$ ($k=1,2,\cdots k_1+k_2$) by the role that emanating from a point $z^-_{E,k}$, the orbit reaches to the point $z^+_{E,k+1}$ after time $\Delta t^-_{E,k}$, then to the point $z^-_{E,k+1}$ and so on. Note that $\Delta t^-_{E,k}$ remains bounded uniformly for any $E>0$.  Restricted on small neighborhoods of these points, denoted by $B_d(\bar z^{\pm}_{E,k})$, the flow $\Phi_{G_0}^t$ defines a local diffeomorphism $\Psi_{E,\delta}$: $\Sigma^{-}_{E,\delta}\supset B_d(\bar z^{-}_{E,k})\to \Sigma^{+}_{E,\delta}$. Because of the smooth dependence of ODE solutions on initial data, a small $\varepsilon>0$ exists such that, for the vector $v^{\pm}$ $\varepsilon$-parallel to $T_{z^{\pm}_{\delta}}(W^{\pm}\cap\Sigma^{\pm}_{0,\delta})$ in the sense that $|\langle v^{\pm},v^{*\pm}\rangle|\ge (1-\varepsilon) \|v^{\pm}\|\|v^{*\pm}\|$ holds for some $v^{*\pm}\in T_{z^{\pm}_{\delta}}(W^{\pm}\cap\Sigma^{\pm}_{0,\delta})$, we obtain from the hyperbolicity of $\Psi_{0,\delta}$ (see the formulae above the figure) that
$$
C_5^{-1}\left(\frac D{\delta}\right)^{2(\frac {\lambda_2}{\lambda_1}-\mu_3)}\le \frac {\|D\Psi_{E,\delta}(z^-_{E,k})v^-\|}{\|v^-\|}
\le C_5\left(\frac D{\delta}\right)^{2(\frac {\lambda_2}{\lambda_1}+\mu_3)},
$$
and
$$
C_5^{-1}\left(\frac D{\delta}\right)^{2(\frac {\lambda_2}{\lambda_1}-\mu_3)}\le \frac {\|D\Psi_{E,\delta}^{-1}(z^+_{E,k})v^+\|}{\|v^+\|}
\le C_5\left(\frac D{\delta}\right) ^{2(\frac {\lambda_2}{\lambda_1}+\mu_3)}
$$
where $C_5\ge C_4>1$, $0<\mu_3\to 0$ as $D\to 0$. If the vector $v^-$ is chosen  $\varepsilon$-parallel to $T_{z^{-}_{\delta}}(W^{-}\cap\Sigma^{-}_{0,\delta})$ then the vector $D\Psi_{E,\delta}(z^-_{E,k})v^-$ is $\varepsilon$-parallel to $T_{z^{+}_{\delta}}(W^{+}\cap\Sigma^{+}_{0,\delta})$.

For $E>0$, the Hamiltonian flow $\Phi_{G_0}^t$ defines local diffeomorphism $\Psi^+_{E,\delta,\delta}$: $\Sigma^{+}_{E,\delta}\supset B_d(\bar z^{+}_{E,k})\to\Sigma^{-}_{E,\delta}$. To make sure $\Psi^+_{E,\delta,\delta}(B_d(\bar z^{+}_{E,k}))\subset\Sigma^{-}_{E,\delta}$ one has $d\to 0$ as $E\to 0$. According to Formula (\ref{regularenergyeq2}), starting from $\Sigma^{+}_{E,\delta}$, the periodic orbit comes to $\Sigma^{-}_{E,\delta}$ after a time approximately equal to
$$
T=\frac 1{\lambda_1}\Big|\ln\Big(\frac{\delta^2}{E}\Big)\Big|+\tau_{\delta}
$$
in which $\tau_{\delta}$ is uniformly bounded as $\delta\to 0$. Given a vector $v$, we use $v_i$ denote the $(x_i,y_i)$-component. For a vector $v^+$ $\varepsilon$-parallel to $T_{z^+_{0,\delta}}(W^-\cap\Sigma^+_{0,\delta})$, there is $C>0$ such that $\|v^+_2\|\ge C\|v^+_1\|$. From Eq.(\ref{cylindereq4}) one obtains
\begin{equation}\label{cylindereq5}
\begin{aligned}
\|v^+_2\|e^{(\lambda_2-\mu)T}\le&\|D\Psi^+_{E,\delta,\delta}(z^+_{E,k})v^+_2\|\le\|v^+_2\|e^{(\lambda_2+ \mu)T},\\
\|v^+_1\|e^{(\lambda_1-\mu)T}\le&\|D\Psi^+_{E,\delta,\delta}(z^+_{E,k})v^+_1\|\le\|v^+_1\|e^{(\lambda_1+ \mu)T}
\end{aligned}
\end{equation}
where $0<\mu\to 0$ as $\delta\to0$, $\lambda_2>\lambda_1>0$. It follows that the vector $D\Psi^+_{E,\delta,\delta}(z^+_{E,k})v^+$ is $\varepsilon$-parallel to $T_{z^-_{0,\delta}}(W^-\cap\Sigma^-_{\delta})$ and
$$
C_6^{-1}\Big(\frac {\delta^2}E\Big)^{\frac{\lambda_2}{\lambda_1}-\mu_4}
\le\frac{\|D\Psi^+_{E,\delta,\delta}(z^+_{E,k})v^+\|}{\|v^+\|}
\le C_6\Big(\frac {\delta^2}E\Big)^{\frac{\lambda_2}{\lambda_1}+\mu_4}
$$
where $C_6>1$ and $\mu_4\downarrow 0$ as $\delta\downarrow 0$. Similarly, for a vector $v^-$ $\varepsilon$-parallel to $T_{z^-_{0,\delta}}(W^+\cap\Sigma^-_{0,\delta})$, one sees that the vector $D{\Psi^+_{E,\delta,\delta}(z^-_{E,j})}^{-1}v^-$ is $\varepsilon$-parallel to $T_{z^-_{0,\delta}}(W^-\cap\Sigma^-_{0,\delta})$ and
$$
C_6^{-1}\Big(\frac {\delta^2}E\Big)^{\frac{\lambda_2}{\lambda_1}-\mu_4}
\le\frac{\|D{\Psi^+_{E,\delta,\delta}}^{-1}(z^-_{E,k})v^-\|}{\|v^-\|}
\le C_6\Big(\frac {\delta^2}E\Big)^{\frac{\lambda_2}{\lambda_1}+\mu_4}.
$$

The composition of these two maps makes up a Poin\'care map $\Phi_{E,\delta}=\Psi^+_{E,\delta,\delta}\circ\Psi_{E,\delta}$, it maps a small neighborhood of the point $z^-_{E,k}$ in $\Sigma^{-}_{E,\delta}$ to a small neighborhood of the point $z^-_{E,k+1}$ in $\Sigma^{-}_{E,\delta}$. For a vector $v^-$ $\varepsilon$-parallel to $T_{z^-_{0,\delta}}(W^-\cap\Sigma^-_{0,\delta})$ the vector $D\Phi_{E,\delta}(z^-_{E,k})v^-$ is still $\varepsilon$-parallel to $T_{z^-_{0,\delta}}(W^-\cap\Sigma^-_{0,\delta})$ and
\begin{equation}\label{cylindereq6}
\Lambda^{-1}\left(\frac {D^2}{E}\right)^{\frac{\lambda_2}{\lambda_1}-\mu_5}
\le\frac {\|D\Phi_{E,\delta}(z^-_{E,k})v^-\|}{\|v^-\|}
\le \Lambda\left(\frac{D^2}{E}\right)^{\frac {\lambda_2}{\lambda_1}+\mu_5},
\end{equation}
and for a vector $v^+$ $\varepsilon$-parallel to $T_{z^-_{0,\delta}}(W^+\cap\Sigma^-_{0,\delta})$ the vector $D\Phi_{E,\delta}^{-1}(z^-_{E,k})v^+$ is still $\varepsilon$-parallel to $T_{z^-_{0,\delta}}(W^+\cap\Sigma^-_{0,\delta})$ and
\begin{equation}\label{cylindereq7}
\Lambda^{-1}\left(\frac {D^2}{E}\right)^{\frac{\lambda_2}{\lambda_1}-\mu_5}
\le\frac{\|D\Phi_{E,\delta}^{-1}(z^-_{E,k})v^+\|}{\|v^+\|}
\le \Lambda\left(\frac {D^2}{E}\right)^{\frac{\lambda_2}{\lambda_1}+\mu_5}
\end{equation}
holds for each $k$, where $\Lambda\ge C_5C_6>1$, $0<\mu_5\to 0$ as $D\to 0$. Therefore, each point $z^-_{E,k}$ is a hyperbolic fixed point for the map $\Phi^{k_i+k_{i+1}}_{E,\delta}$, $\{z^-_{E,k}:k=1,\cdots,k_i+k_{i+1}\}$ is a hyperbolic orbit of $\Phi_{E,\delta}$. It will be proved in \cite{C15} that these points are uniquely ordered, $k_i+k_{i+1}$ is the minimal period. We complete the proof.
\end{proof}
\begin{cor}
The $(E,g)$-minimal periodic orbit in the energy level $G_0^{-1}(E)$ with $E\le E'$ has a continuation of hyperbolic periodic orbits which approach to that two homoclinic orbits $d\gamma_1$ and $d\gamma_2$. They make up an invariant cylinder which takes the homoclinic orbits as its boundary.
\end{cor}
\begin{proof}
According to Lemma \ref{cylinderlem1}, the hyperbolicity of $(E,g)$-minimal orbit becomes very strong when $E\downarrow 0$. Such hyperbolic property is gained if the periodic orbit approaches the homoclinic orbits, the minimal property is not used. By the theorem of implicit function, this $(E,g)$-minimal orbit has a continuation of hyperbolic periodic orbits arbitrarily close to the homoclinic orbits $d\gamma_1$ and $d\gamma_2$.
\end{proof}

Let $E'_1=h(p'_1)$. As we increase the energy from $E'$ to $E'_1$, it follows from Theorem \ref{mainth} that there are finitely many $E_i\in [E',E'_1]$ only such that for $E\in [E',E'_1]\backslash\{E_i\}$, the energy level $G_0^{-1}(E)$ contains only one $(E,g)$-minimal orbit and $G_0^{-1}(E_i)$ contains two minimal periodic orbits. We call these $\{E_i\}$ bifurcation points. Therefore, these hyperbolic orbits make up finitely many pieces of invariant cylinder, normally hyperbolic for the time-$2\pi$-map $\Phi_{G_0}^{2\pi}$, produced by the Hamiltonian flow $\Phi_{G_0}^t$.

In the next step, we are going to study if these cylinders can survive the map $\Phi_{G_{\epsilon}}$, induced by the flow $\Phi_{ G_{\epsilon}}^t$, where $G_{\epsilon}$, defined in \ref{Hamilton2}, is a small time-periodic perturbation of $G_0$.

\subsection{Invariant splitting of the tangent bundle: near double resonance}
As shown in the last section, there is a cylinder made up by periodic orbits $(x_{E}(t),y_{E}(t))$ of $\Phi_{G_0}^t$ which extends from the energy level $G_0^{-1}(E')$ to the homoclinic orbits, denoted by
$$
\Pi_{0,E',g}=\{(x_{E}(t),y_{E}(t)):[x_{E}]=g,E\in (0,E'],t\in\mathbb{R}\}.
$$
Let $T(E)$ denote the period of the periodic orbit in $G_0^{-1}(E)$, for any $0<a<b\le E'$ one has
$$
\int_{\Pi_{a,b,g}}\omega=\int_{a}^{b}\int_{0}^{T(E)}dE\wedge dt>0.
$$
The cylinder might be slant and crumpled, we want to know how the symplectic area is related to the usual area of the cylinder. Since the cylinder is made up by periodic orbits, if we denote the intersection point of the orbit $z_E(t)$ with the section $x_1=\delta$ by $(\delta,y_1(E),x_2(E),y_2(E))$, then $(x_2(E),y_2(E))$ is a fixed point of the Poincar\'e return map $\Phi_{E,\delta}$, i.e. $\Phi_{E,\delta}(x_2(E),y_2(E))=(x_2(E),y_2(E))$.
So we have
\begin{equation}\label{cylindereq8}
\Big(\frac{\partial^2\Phi_{E,\delta}}{\partial x_2\partial y_2}-\mathrm{id}\Big)\Big(\frac{\partial x_2}{\partial y_1},\frac{\partial y_2}{\partial y_1}\Big)^t=\frac{\partial\Phi_{E,\delta}}{\partial y_1}.
\end{equation}
To study the quantity $\frac{\partial\Phi_{E,\delta}}{\partial y_1}$, let us recall the picture of Figure \ref{fig7}. Emanating from a point $(\delta,y_1,x_2,y_2)\in G_0^{-1}(E)$ the orbit reach a point $z$ in the section $\{x_1=-\delta\}$ after a time $\tau(E,\delta)$. Let $z^*\in\{x_1=-\delta\}$ be the point corresponding to $(\delta,y^*_1,x_2,y_2)\in \bar G^{-1}(E^*)$, obtained in the same way. Since $\tau(E,\delta)$ remains bounded as $E\downarrow 0$, the difference of the $(x_2,y_2)$-coordinate of $z$ and $z^*$ is bounded by $d_0|y_1-y^*_1|$ where $d_0$ depends on $\tau(E,\delta)$.  Let $(\Delta x, \Delta y)$ be the solution of the variational equation (\ref{cylindereq4}) along the $(E,g)$-minimal periodic solution $(x_{E}(t),y_{E}(t))$, let $t_0<t_1$ be the time such that $x_{E,1}(t_0)=-\delta$ and $x_{E,1}(t_1)=\delta$ if we use the notation $x_E=(x_{E,1},x_{E,2})$, the quantity $t_1-t_0$ is bounded by (\ref{regularenergyeq2}). In virtue of the formula (\ref{cylindereq5}), one obtains
\begin{align*}
\|(\Delta x, \Delta y)(t_1)\|&\le C_7e^{(\lambda_2+\mu)T}\|(\Delta x, \Delta y)(t_0)\|\\
&\le C_8E^{-\frac{\lambda_2}{\lambda_1}-\mu_6}\|(\Delta x, \Delta y)(t_0)\|
\end{align*}
where $0<\mu_6\to 0$ as $\delta\to 0$. It implies that
$$
\Big\|\frac{\partial\Phi_{E,\delta}}{\partial y_1}\Big\|\le C_8E^{-\frac{\lambda_2}{\lambda_1}-\mu_6}.
$$
Since Hamiltonian flow preserves the symplectic structure, the matrix $\frac{\partial\Phi_{E,\delta}}{\partial (x_2,p_2)}$ is area-preserving. One eigenvalue is large, lower bounded by (\ref{cylindereq6}), another one is small. Let $\zeta_1$ be the eigenvector for the large eigenvalue and $\zeta_2$ be the eigenvector for the small one, then we have a decomposition of
$$
\frac{(\Delta x, \Delta y)(t_1)}{\|(\Delta x, \Delta y)(t_1)\|}=a_1\zeta_1+a_2\zeta_2.
$$
Because of the hyperbolic structure (see (\ref{cylindereq4}), (\ref{cylindereq6}) and (\ref{cylindereq7})), if
$$
\|(\Delta x, \Delta y)(t_1)\|\ge C_9\|(\Delta x, \Delta y)(t_0)\|,
$$
the projection of vector $(\Delta x, \Delta y)(t_1)$ to $\zeta_1$ is not small in the sense that $|a_1|\ge\frac 12$. Where $C_9>0$ is suitably large, but independent of small $E>0$. Therefore, we obtain from Equation (\ref{cylindereq8}) that
\begin{equation*}\label{cylindereq9}
\Big\|\frac{\partial x_2}{\partial y_1}\Big\|,\Big\|\frac{\partial y_2}{\partial y_1}\Big\|\le C_{10}E^{-2\mu_6}.
\end{equation*}
It provides a lower bound of the symplectic area $\omega$ with respect to the usual area $S$ of the cylinder $\Pi_{a,b,g}$
\begin{equation}\label{cylindereq10}
|\omega|\ge C_{11}E^{2\mu_6}|S|.
\end{equation}

To study the invariant splitting of the tangent bundle over the cylinder $\Pi_{0,E',g}$, for $E>0$ we define
\begin{equation*}\label{passtime}
T_{E}=\frac 2{\lambda_1}|\ln E|.
\end{equation*}
\begin{theo}\label{cylinderthm1}
With the hypotheses \text{\rm ({\bf H1}), ({\bf H2})} and $0<E_d\le E'$, the invariant cylinder $\Pi_{E_d,E',g}$ is normally hyperbolic for the map $\Phi_{G_0}^s$, where $s\ge T_{E_d}$. The tangent bundle of $\mathbb{T}^2$ over $\Pi_{E_d,E',g}$ admits the invariant splitting:
$$
T_zM=T_zN^+\oplus T_z\Pi_{E_d,E',g}\oplus T_zN^-
$$
some $\Lambda_1\ge 1$, $\Lambda_2\ge 1$ and small $E>0$ exist such that  $\lambda_2/\lambda_1-\nu>1+\nu$
\begin{equation}\label{cylindereq11}
\begin{aligned}
\Lambda_1^{-1}E_d^{1+\nu}<\frac {\|D\Phi^s_{ G_0}(z)v\|}{\|v\|}&<\Lambda_1E_d^{-1-\nu},\qquad \forall\ v\in T_z\Pi_{E_d,E',g},\\
\frac{\|D\Phi_{ G_0}^{s}(z)v\|}{\|v\|}&\le \Lambda_2 E_d^{\frac{\lambda_2}{\lambda_1}-\nu},\qquad \forall\ v\in T_zN^+,\\
\frac{\|D\Phi_{ G_0}^{s}(z)v\|}{\|v\|}&\ge\Lambda_2^{-1} E_d^{-\frac{\lambda_2}{\lambda_1}+\nu}, \qquad \forall\ v\in T_zN^-.
\end{aligned}
\end{equation}
\end{theo}
\begin{proof}
The cylinder $\Pi_{0,E',g}$ is a 2-dimensional symplectic sub-manifold, invariant for the Hamiltonian flow $\Phi^s_{ G_0}$.  However, it is not clear whether this cylinder admits the invariant splitting so that Formula (\ref{cylindereq11}) holds for the time-$2\pi$-map $\Phi_{ G_0}=\Phi^t_{ G_0}|_{t=2\pi}$. It is possible that
\begin{align*}
m(D\Phi_{ G_0}|_{T\Pi_{E_i,E_{i+1},g}})=&\inf\{|D\Phi_{ G_0}v|:v\in T\Pi_{E_i,E_{i+1},g}, |v|=1\}<1,\\
&\|D\Phi_{ G_0}|_{T\Pi_{E_i,E_{i+1},g}}\|>1,
\end{align*}
and we do not know the norm of $D\Phi_{ G_0}$ when it acts on the normal bundle.

From Formulae (\ref{cylindereq6}) and (\ref{cylindereq7}), one sees that the smaller the energy reaches, the stronger hyperbolicity the map $\Phi_{E,\delta}$ obtains. The strong hyperbolicity is obtained by passing through small neighborhood of the fixed point. However, on the other hand, the smaller the energy decreases, the longer the return time becomes.

Let $\Delta t_{E,k}$ denote the time interval such that, starting from $z^-_{E,k}$, the periodic orbit comes to $z^-_{E,k+1}$ after time $\Delta t_{E,k}$, then $\Delta t_{E,k}=\tau_{E,k}-\frac 1{\lambda_1}\ln E$
where $\tau_{E,k}$ is uniformly bounded (see Formula (\ref{regularenergyeq2})). For small $E>0$, emanating from any point $z$ on the minimal periodic orbit $z_{E}(s)$ and after a time $T_E$ of (\ref{passtime}), $\Phi_{ G_0}^s(z)$ passes through a neighborhood of the fixed point at least once. Therefore, the map $\Phi_{ G_0}^s|_{s\ge T_{E}}$ obtains strong hyperbolicity on normal bundle such as (\ref{cylindereq6}) and (\ref{cylindereq7}).

To see how the map $D\Phi^s_{G_0}$ acts on the tangent bundle, let us study how it elongates or shortens small arc of the periodic orbit $z_{E}(t)$. To pass through $\delta$-neighborhood of the origin along the orbit $z_E(t)$, it needs a time approximately equal to $|\lambda_1^{-1}\ln \delta^{-2}E|$. Restricted in $\delta$-neighborhood of the origin, there exists small $\mu_7>0$ such that
$$
|x(0)|e^{-(\lambda_1+\mu_7)t}\le|x(t)|\le |x(0)|e^{(\lambda_1+\mu_7)t},
$$
Therefore, the variation of the length of short arc is between $O(E^{1+\mu_7})$ and $O(E^{-1-\mu_7})$. Because of the relation between the symplectic area $\omega$ and the usual area $S$ of the cylinder, provided by the formula (\ref{cylindereq10}), the variation of $\|D\Phi^s_{ G_0}\|$, restricted on the tangent bundle of the cylinder, is between $O(E^{1+\mu_7+2\mu_6})$ and $O(E^{-1-\mu_7-2\mu_6})$, where we use the property that Hamiltonian flow preserves the symplectic structure. Due to periodicity, this lower and upper bound is independent of $s$. Therefore, the theorem is proved.
\end{proof}

For cylinder $\Pi_{E_i,E_{i+1},g}$ with $E_i\ge E'$, the normal hyperbolicity is obvious.
\begin{theo}
For $E'\le E_i<E_{i+1}\le E'_1$ and typical $V$, there exists $s_0>0$ depending on $E'$, $E'_1$ and $V$, such that the tangent bundle over the invariant cylinder $\Pi_{E_i,E_{i+1},g}$ admits $D\Phi^s_{ G_0}$-invariant splitting
$$
T_zM=T_zN^+\oplus T_z\Pi_{E_i,E_{i+1},g}\oplus T_zN^-
$$
some $\Lambda_2>\Lambda_1\ge 1$ such that the following hold for $s\ge s_0$
\begin{equation}\label{cylindereq12}
\begin{aligned}
\Lambda_1^{-1}<\frac {\|D\Phi^s_{ G_0}(z)v\|}{\|v\|}&<\Lambda_1,\qquad \forall\ v\in T_z\Pi_{E_i,E_{i+1},g},\\
\frac{\|D\Phi_{ G_0}^{s}(z)v\|}{\|v\|}&\le \Lambda_2,\qquad \forall\ v\in T_zN^+,\\
\frac{\|D\Phi_{ G_0}^{s}(z)v\|}{\|v\|}&\ge\Lambda_2^{-1}, \qquad \forall\ v\in T_zN^-.
\end{aligned}
\end{equation}
\end{theo}
\begin{proof} The cylinder is a symplectic sub-manifold, made up by minimal periodic orbits. Therefore, some $\Lambda_1\ge 1$ exists such that
$$
\Lambda^{-1}\|v\|\le\|\Phi^s_{ G_0}(z_E(t))v\|\le\Lambda\|v\|
$$
holds for any $s>0$ if $v$ is a vector tangent to $z_E$ at $z_E(t)$. Since the Hamiltonian flow preserves the symplectic form $\omega$, restricted on the cylinder which is an area element. Clearly, $|\omega|$ is lower bounded by usual area element $|S|$. It follows that the last formula holds for any vector tangent to the cylinder at $z_E(t)$. It verifies the first formula in (\ref{cylindereq11}). Let $\Sigma_{E,z}\subset G_0^{-1}(E)$ be a two-dimensional disk, transversally intersects the periodic orbit $z_E(t)$. The flow $\Phi_{ G_0}^t$ defines a Poincar\'e return map $\Phi_{E}$, the fixed point corresponds to the periodic orbit. Let $\lambda_{1,E}$ and $\lambda_{2,E}$ be the eigenvalues of the matrix $D\Phi_{E}$, it depends on the energy $E$. According to Theorem \ref{mainth}, each of these orbits is hyperbolic, namely, some $\lambda>1$ exists such that
$$
\min\{|\lambda_{1,E}|,|\lambda_{2,E}|\}\le\lambda^{-1}<\lambda\le\max\{|\lambda_{1,E}|,|\lambda_{2,E}|\},\qquad \forall\ E\in[E_i,E_{i+1}].
$$
Let $\Lambda_2=\lambda([\frac{\Lambda_1}{\lambda}]+2)$, then $\Lambda_2>\Lambda_1$. Let $T_E$ be the period of the orbit $z_E(t)$ and set
\begin{equation}\label{returntime}
s_i=\max_{E\in[E_i,E_{i+1}]}T_E\Big(\Big[\frac{\Lambda_1}{\lambda}\Big]+2\Big),
\end{equation}
the second and the third formulae in (\ref{cylindereq11}) holds for $\Phi_{ G_0}^s$ with $s\ge s_i$.
\end{proof}

The cylinder $\Pi_{0,E',g}$ may extend to the energy level $ G_0^{-1}(E_1+\Delta)$, where $\Pi_{E',E_1,g}$ is made up by $(E,g)$-minimal orbits, but Formula (\ref{cylindereq12}) applies to $\Pi_{E',E_1+\Delta,g}$. One can see that the whole cylinder $\Pi_{E_d,E_1+\Delta,g}$ is normal hyperbolic for $D\Phi_{ G_0}^s$ for $s\ge\max\{T_{E_d},s'\}$ where $s'$ is defined so that (\ref{cylindereq12}) holds for $\Pi_{E',E_1+\Delta,g}$ (cf. (\ref{returntime})).

\subsection{Bifurcation point}
Let $E_i<E_{i+1}$ be two adjacent bifurcation points, then each $G_0^{-1}(E)$ contains only one $(E,g)$-minimal orbit for $E\in(E_i,E_{i+1})$, denoted by $z_E$. Let $z^-_{E_i}=\lim_{E\downarrow E_i}z_E$, $z^+_{E_{i+1}}=\lim_{E\uparrow E_{i+1}}z_E$. These orbits make up an invariant cylinder
$$
\Pi_{E_i,E_{i+1},g}=\{(x_{E}(t),y_{E}(t)):[x_{E}]=g,E\in [E_i,E_{i+1}],t\in\mathbb{R}\}.
$$
By definition, at the  bifurcation point $E_i$, there are two minimal periodic orbits in typical case, denoted by $z_{E_i}^+(t)$ and $z_{E_i}^-(t)$ in the energy level $ G_0^{-1}(E_i)$. The orbit $z_{E_i}^+(t)$ makes up the upper boundary of $\Pi_{E_{i-1},E_i,g}$ and the orbit $z_{E_i}^-(t)$ makes up the lower boundary of $\Pi_{E_{i},E_{i+1},g}$.
Because of the implicit function theorem, there is continuation of hyperbolic periodic orbits which extends from $z_{E_i}^{+}(t)$ to higher energy, denoted by $z_E^+(t)$, also hyperbolic orbits extending from $z_{E_i}^{-}(t)$ to lower energy, denoted by $z_E^-(t)$. Those $z_{E}^{\pm}(t)$ are not in the Mather set unless $E=E_i$, the action along these orbits reaches local minimum instead of global minimum. In this way, we have the cylinder $\Pi_{E_{i-1}-\Delta,E_i+\Delta,g}$ which ranges from the energy level $ G_0^{-1}(E_{i-1}-\Delta)$ to the energy level $ G_0^{-1}(E_i+\Delta)$, as well as the cylinder $\Pi_{E_{i}-\Delta,E_{i+1}+\Delta,g}$ which ranges from the energy level $ G_0^{-1}(E_{i}-\Delta)$ to the energy level $ G_0^{-1}(E_{i+1}+\Delta)$. The normally hyperbolic invariant splitting (\ref{cylindereq12}) applies to the extended cylinders $\Pi_{E_{i-1}-\Delta,E_i+\Delta,g}$.

By the definition we have $\frac{\partial F}{\partial E}(x^+_{E_i}(0),E_i)\ge\frac{\partial F}{\partial E}(x^-_{E_i}(0),E_i)$. It is obviously a generic condition that

({\bf H3}). $\frac{\partial F}{\partial E}(x^+_{E_i}(0),E_i)>\frac{\partial F}{\partial E}(x^-_{E_i}(0),E_i)$.

\subsection{Persistence of NHICs: near double resonance}
We apply the theorem of normally hyperbolic manifold \cite{HPS} to obtain NHIC for the Hamiltonian $G_{\epsilon,0}$ of (\ref{Hamilton2}). We need the following preliminary lemma.

\begin{lem}\label{cylinderlem2}
Let the equation $\dot z=F_{\epsilon}(z,t)$ be a small perturbation of $\dot z=F_0(z,t)$, let $\Phi_{\epsilon}^t$ and $\Phi_0^t$ denote the flow determined by these two equations respectively. Then
$$
\|\Phi_{\epsilon}^t-\Phi_{0}^t\|_{C^1}\le \frac BA(1-e^{-At})e^{2At}
$$
where $A=\max_{t,\lambda=\epsilon,0}\|F_{\lambda}(\cdot,t)\|_{C^2}$ and $B=\max_t\|(F_{\epsilon}-F_0)(\cdot,t)\|_{C^1}$.
\end{lem}
\begin{proof} Let $z_{\lambda}(t)$ denote the solution of the equations $\dot z=F_{\lambda}(z,t)$ for $\lambda =\epsilon,0$ respectively, and $z_{\epsilon}(0)=z(0)$. Let $\Delta z(t)=z_{\epsilon}(t)-z(t)$, then $\Delta z(0)=0$ and
$$
\Delta\dot z=\partial_zF_{\epsilon}((\nu z+(1-\nu)z_{\epsilon})(t),t)\Delta z+(F_{\epsilon}-F_0)(z(t),t)
$$
where $\nu=\nu(t)\in [0,1]$. Therefore, one has
$$
\|\Delta\dot z\|\le\max\|\partial_zF_{\epsilon}\|\|\Delta z\|+\max\|F_{\epsilon}-F_0\|.
$$
Let $\Delta z=y-\frac BA$, we have $\dot y\le Ay$. It follows from Gronwell's inequality that
$$
\|\Delta z(t)\|\le\frac BA(e^{At}-1).
$$

Along the orbit $z_{\lambda}(t)$, the differential of the flow $\Phi^t_{\lambda}$  obviously satisfies the equation
$$
\frac d{dt}D\Phi^t_{\lambda}=\partial_zF_{\lambda}(z_{\lambda}(t),t)D\Phi^t_{\lambda},\qquad \lambda=\epsilon, 0.
$$
Therefore, for each tangent vector $v$ attached to $z_{\lambda}(0)$ one has
$$
\|D\Phi^t_{\lambda}v\|\le\|v\|e^{At}.
$$

To study the differential of $\Phi_{\epsilon}^t-\Phi_{0}^t$, let us consider the equation of secondary variation. Let $\delta z_{\lambda}$ be the solution of the variational equation $\delta\dot z_{\lambda}=\partial_zF_{\lambda}(z_{\lambda}(t),t)\delta z_{\lambda}$ for $\lambda=\epsilon,0$ respectively, where $z_{\lambda}(t)$ solves the equation $\dot z_{\lambda}= F_{\lambda}(z_{\lambda},t)$ and $z_{\epsilon}(0)=z(0)$.  To measure the size $\Delta\delta z=\delta z_{\epsilon}-\delta z$ with the condition $z_{\epsilon}(0)=z(0)$, we make use of the relations such as $v=\delta z_{\epsilon}(0)=\delta z(0)$, $\|\delta z(t)\|\le\|v\|e^{At}$ and find that
\begin{align*}
\Big\|\frac {d(\Delta\delta z)}{dt}\Big\|\le &\max\|\partial _zF_{\epsilon}\|\|\Delta\delta z\| +\max\|\partial^2_zF\|\|\Delta z(t)\|\|\delta z(t)\| \\
&+\max\|\partial_z(F_{\epsilon}-F)\|\|\delta z(t)\| \\
\le &A\Delta\delta z+B\|v\|e^{2At}.
\end{align*}
Let $\Delta\delta z=y+\frac BA\|v\|e^{2At}$, we have $\dot y\le Ay$. Using Gronwell's inequality again, one obtains an upper bound of the variation of the differential
$$
\|\Delta\delta z(t)\|\le\frac BA\|v\|(1-e^{-At})e^{2At}.
$$
Note that $v$ represents initial tangent vector, it completes the proof.
\end{proof}
Let us apply this lemma to study the invariant cylinders of the Hamiltonian $G_{\epsilon,0}$. A sub-manifold $N$ is called overflowing invariant for a flow $\Phi^s$ if, for each $z\in\mathrm{int}N$, the orbit $\Phi^s(z)$ either stays in $N$ forever, or by passing through $\partial N$ to leave. We use $\Phi^{s,s_0}$ to denote the map from the time $s_0$-section to the time $s$-section. A sub-manifold $N'$ is called a $\delta$-deformation of another sub-manifold $N$ if $d_H(N,N')\le\delta$, where $d_H$ denotes Hausdorff distance.
\begin{theo}\label{cylinderforepsilon}
In the extended phase space $T^*\mathbb{T}^2\times\sqrt{\epsilon}\mathbb{T}$, the Hamiltonian flow $\Phi^{s}_{G_{\epsilon,0}}$ admits overflowing invariant cylinders $\tilde\Pi_{\epsilon^d,E_1+\Delta-\epsilon^d,g}$, $\tilde\Pi_{E_i-\Delta+\epsilon^d,E_{i+1}+\Delta-\epsilon^d,g}$ which are the $\epsilon^{\sigma}$-deformation of the cylinder $\Pi_{\epsilon^d,E_1+\Delta-\epsilon^d,g} \times\sqrt{\epsilon}\mathbb{T}$, $\Pi_{E_i-\Delta+\epsilon^d,E_{i+1}+\Delta-\epsilon^d,g} \times\sqrt{\epsilon}\mathbb{T}$ respectively, if
\begin{equation}\label{d>0}
0<d<\min\Big\{\frac{\lambda_1}{12\max_x\sqrt{\|A\|^2+\|\partial^2V\|^2}},\frac 14\Big\}\sigma
\end{equation}
and $\epsilon\ge 0$ is sufficiently small. The cylinder $\tilde\Pi_{\epsilon^d,E_1+\Delta-\epsilon^d,g}$ admits normally hyperbolic invariant splitting of $($\ref{cylindereq11}$)$ for the map $\Phi^{s,s_0}_{G_{\epsilon}}$ with $s-s_0=\frac 2{\lambda_1}|\ln\epsilon^{3d}|$; the cylinder $\tilde\Pi_{E_i-\Delta+\epsilon^d,E_{i+1}+\Delta-\epsilon^d,g}$ admits normally hyperbolic invariant splitting of $($\ref{cylindereq12}$)$ for $\Phi^{s,s_0}_{G_{\epsilon}}$ where $s-s_0$ is given by $($\ref{returntime}$)$, independent of $\epsilon$.
\end{theo}
\begin{proof}
Considering $R_{\epsilon,0}$ as the function of $(x,y)$ and treating $\theta$ as parameter, we find that there exists some constant $C_{11}=\max_{\theta}\|R_{\epsilon,0}(\cdot,\theta)\|_{C^1}$ such that
$$
\max_{\theta}\|J\nabla G_0-J\nabla G_{\epsilon,0}\|_{C^1}\le C_{11}\epsilon^{\sigma}.
$$
Let $C_{12}=\max_x\sqrt{\|A\|^2+\|\partial^2V\|^2}$, for $s-s_0=\frac 2{\lambda_1}|\ln\epsilon^{3d}|$ one obtains from Lemma \ref{cylinderlem2} that
$$
\|\Phi^{s,s_0}_{ G_0}-\Phi^{s,s_0}_{G_{\epsilon,0}}\|_{C^1}\le \frac{C_{11}}{C_{12}}\epsilon^{\sigma-\frac {12C_{12}d}{\lambda_1}}.
$$
If the condition $0<d<\frac{\lambda_1\sigma}{12C_{12}}$ holds, then $\|\Phi^{s,s_0}_{ G_0} -\Phi^{s,s_0}_{G_{\epsilon,0}}\|_{C^1}\to 0$ as $\epsilon\to 0$. It allows one to apply the theorem of normally hyperbolic manifold to obtain the existence of invariant cylinder:

We consider a piece of hyperbolic cylinder $\Pi_{E_d,E_1+\Delta,g}\subset\Pi_{0,E_1+\Delta,g}$ with $E_d=\epsilon^{3d}$ for some small $d>0$. Since these hyperbolic properties are posed for the map $\Phi^{s,s_0}_{ G_0}$ with large $s-s_0$, one has to measure how large the quantity $\|\Phi_{ G_0}^{s,s_0}-\Phi_{G_{\epsilon,0}}^{s,s_0}\|$ will be. As $ G_0$ is autonomous, $\Phi_{ G_0}^{s,s_0}=\Phi_{ G_0}^{s-s_0}$.

Note that $\Pi_{E_d,E_1+\Delta,g}$ is a cylinder with boundary, normally hyperbolic and invariant for $\Phi^s_{ G_0}$, with $s=\frac 2{\lambda_1}|\ln \epsilon^{3d}|$, we do not expect that the whole cylinder survives small perturbation, it may lose some part close to the boundary. As the first step to measure to what range the cylinder survives, we modify the Hamiltonian $G_{\epsilon,0}$. Let $\rho$ be $C^2$-function such that $\rho(\mu)=1$ for $\mu\ge 1$ and $\rho(\mu)=0$ for $\mu\le 0$. Let $E_{1}^{\Delta}=E_1+\Delta$, we introduce
\begin{equation}\label{modification}
G'_{\epsilon}=\begin{cases}
G_0(x,y)+\epsilon^{\sigma}\rho_1(x,y)R_{\epsilon,0}(x,y,\theta), &\mathrm{if}\  G_0(x,y)\in[\epsilon^{3d},\frac 12\epsilon^d],\\
G_0(x,y)+\epsilon^{\sigma}\rho_2(x,y)R_{\epsilon,0}(x,y,\theta),&\mathrm{if}\  G_0(x,y)\in[E_{1}^{\Delta}-\epsilon^d,E_{1}^{\Delta}],\\
G_{\epsilon,0}(x,y,\theta), &\mathrm{elsewhere}
\end{cases}
\end{equation}
where $\rho_1(x,y)=\rho(2\frac{ G_0(x,y)-\epsilon^{3d}}{\epsilon^d-2\epsilon^{3d}})$, $\rho_2=1-\rho(\frac{ G_0(x,y)-(E_1+\Delta-\epsilon^d)}{\epsilon^d})$. So, $\|G'_{\epsilon,0}- G_0\|_{C^2}\ll1$ provided $d<\sigma$ and $\epsilon\ll1$. Clearly, the cylinder $\Pi_{E_d,E_1+\Delta,g}$ survives the perturbation $\Phi^{s,s_0}_{ G_0}\to\Phi^{s,s_0}_{G'_{\epsilon}}$ and the boundary of $\Pi_{E_d,E_1+\Delta,g}$ remains unchanged for $\Phi^{s,s_0}_{G'_{\epsilon}}$. The survived cylinder in the extended phase space $\mathbb{T}^2\times\mathbb{R}^2\times\sqrt{\epsilon}\mathbb{T}$ is denoted by $\tilde\Pi_{E_d,E_1+\Delta,g}$.

Restricted on the cylinder $\tilde\Pi_{E_d,E_1+\Delta,g}\cap\{(x,y,\theta): G_0(x,y)\in[\epsilon^d,E_1+\Delta-\epsilon^d]\}$ as well as $\tilde\Pi_{E_i-\Delta,E_{i+1}+\Delta,g}\cap \{(x,y,\theta): G_0(x,y)\in[E_i-\Delta+\epsilon^d,E_{i+1}+\Delta-\epsilon^d]\}$, we find $G_{\epsilon,0}=G'_{\epsilon}$. It implies the existence of overflowing invariant cylinder $\tilde\Pi_{\epsilon^d,E_1+\Delta-\epsilon^d,g}$ and $\tilde\Pi_{E_i-\Delta+\epsilon^d,E_{i+1}+\Delta-\epsilon^d,g}$ for $\Phi_{G_{\epsilon}}^s$. The normally hyperbolic invariant splitting is an application of the theorem of normally hyperbolic manifold.
\end{proof}

\section{Transition of NHIC from double to single resonance}
\setcounter{equation}{0}
Along the resonant path $\Gamma^*\cap\{|p-p''|\le\epsilon^{\sigma}\}$ we have chosen the points $\{p'_i\}$ such that $p'_0=p''$, $\partial_1 h(p'_i)=Ki\sqrt{\epsilon}$, where $K\in\mathbb{Z}$. There are as many as $2([K^{-1}\epsilon^{\sigma-\frac 12}]+1)$ such points. What we studied in the last section is about the disk which is centered at $p=p''$, the double resonant point. In this section, we consider the disks which are ``quite away from" the double resonance in the sense that $Ki\gg 1$. Let $Ki=\Omega_i$, we rewrite the Hamiltonian of (\ref{2dHmil}) $G_{\epsilon,i}(x,y,\theta)=G_{i}(x,y)+\epsilon^{\sigma}R_{\epsilon,i}(x,y,\theta)$ where
\begin{equation}\label{barGepsilon}
G_{i}(x,y)=\Omega_i y_1+\frac 12\langle Ay,y\rangle-V(x).
\end{equation}
We shall show that the NHICs in the such cube looks more and more like the NHICs in the case of single resonance when the quantity $\Omega_i$ approaches infinity. As the first step, let us consider the Hamiltonian $G_{i}$. Applying Theorem \ref{mainth} proved in \cite{CZ1}, we find that all $(E,g)$-minimal periodic orbits make up some pieces of normally hyperbolic invariant cylinders. However, it is not enough to study the persistence of these NHICs under the small perturbation $G_{i}\to G_{\epsilon,i}=G_{i}+\epsilon^{\sigma}R_{\epsilon,i}$, since the number of cubes approaches infinity as $\epsilon\to 0$. We need to show that, for generic $V$, the hyperbolicity is independent of the numbers.

For a function $V\in C^r(\mathbb{T}^2,\mathbb{R})$, we define
$$
[V](x_2)=\frac 1{2\pi}\int_0^{2\pi}V(x_1,x_2)dx_1.
$$
A set $\mathfrak{V}_{\infty}\subset C^r(\mathbb{T}^2,\mathbb{R})$ is defined so that for each $V\in\mathfrak{V}_{\infty}$, the function $[V]$ has a unique minimal point which is non-degenerate, i.e. $\frac{d^2}{dx_2^2}[V](x_2)>0$ if $x_2$ is the minimal point. Obviously, the set $\mathfrak{V}_{\infty}$ is open-dense in $C^r(\mathbb{T}^2,\mathbb{R})$ with $r\ge 2$.

To denote a cylinder for $G_i$ and for $G_{\epsilon,i}$ respectively, we add superscript $^i$ and $^{i,\epsilon}$ to the usual notation of cylinder $\Pi_{a,b,g}\to \Pi_{a,b,g}^i,\Pi_{a,b,g}^{i,\epsilon}$.

\begin{theo}\label{LmHighEnergy}
Given a potential $V\in\mathfrak{V}_{\infty}$ and a number $\bar K>1$, there exists a suitably large $\Omega^*>0$ such that for $\Omega_i\ge\Omega^*$, the Hamiltonian flow $\Phi^t_{G_{i}}$ of $($\ref{barGepsilon}$)$ admits an invariant cylinder $\Pi_{0,\bar K\Omega_i,g}^i$ made up by $(E,g)$-minimal orbits which lie on the energy level $G_{i}^{-1}(E\Omega_i)$ with $E\in[0,\bar K]$.

Moreover, the tangent bundle of $\mathbb{T}^2$ over $\Pi_{0,\bar K\Omega_i,g}^i$ admits the invariant splitting:
$$
T_zM=T_zN^+\oplus T_z\Pi_{0,\bar K\Omega_i,g}\oplus T_zN^-,
$$
some numbers $\Lambda>\lambda\ge 1$, and an integer $k\ge 1$ exist such that
\begin{equation}\label{wuqiongyuan}
\begin{aligned}
\lambda^{-1}\|v\|<\|D\Phi^{2k\pi}_{G_{i}}(z)v\|&<\lambda\|v\|,\qquad \forall\ v\in T_z\Pi_{0,\bar K\Omega_i,g}^i,\\
\|D\Phi_{G_{i}}^{2k\pi}(z)v\|&\le \Lambda^{-1}\|v\|,\qquad \forall\ v\in T_zN^+,\\
\|D\Phi_{G_{i}}^{2k\pi}(z)v\|&\ge\Lambda\|v\|, \qquad \forall\ v\in T_zN^-.
\end{aligned}
\end{equation}
holds for any large $\Omega_i\ge\Omega^*$.
\end{theo}
\begin{proof}
For large $\Omega_i$, the energy of the Hamiltonian $G_{i}$ ranges over from almost zero to order $O(\Omega_i)$ if $\|y\|\le C$, where $C=O(1)$ is independent of $\Omega_i$. Under the symplectic coordinate transformation
\begin{equation}\label{coordinatetransformation-linear}
(x_1,x_2,y_1,y_2)\to\Big(\frac{x_1}{\Omega_i},x_2,\Omega_i y_1, y_2\Big),
\end{equation}
the Hamiltonian $G_{i}$ turns out to be
\begin{equation}\label{barG}
G'_{i}=y_1+\frac 1{2\Omega^2}A_{11}y_1^2+\frac{A_{12}}{\Omega}y_1y_2+\frac 12A_{22}y_2^2-V\Big(\Omega_i x_1,x_2\Big).
\end{equation}
The equation $G'_{i}(x_1,y_1(x_1,x_2,y_2),y_2)=E\Omega_i$ is solved by the function
\begin{equation}\label{Hamoned}
\begin{aligned}
y_1=&\frac{\Omega_i^2}{A_{11}}\left[-\Big(1+\frac{A_{12}}{\Omega_i}y_2\Big)+\sqrt{\Big(1+\frac{A_{12}} {\Omega_i}y_2\Big)^2-\frac{A_{11}}{\Omega_i^2}(A_{22}y_2^2-2V-2E\Omega_i)}\right] \\
=&E\Omega_i-\frac {1}{2}A_{22}y_2^2-A_{12}Ey_2+V+\Omega_i^{-1}R_H,
\end{aligned}
\end{equation}
where $E$ ranges over an interval $[0,\bar K]$ where $\bar K$ is independent of $\Omega_i$, the remainder $\Omega_i^{-1}R_H$ is of order $O(\Omega_i^{-1})$.

Let $\tau=-x_1$ as the new ``time". The Hamiltonian $y_1$ produces a Lagrangian up to an additive constant
$$
L_1=\frac{1}{2A_{22}}\Big(\frac{dx_2}{d\tau}\Big)^2-\frac{A_{12}E}{A_{22}} \frac{dx_2}{d\tau}+V+\frac 1\Omega_i R_L,
$$
where $R_L$ is $C^{r}$-bounded in $(y_2,\tau,x_2)$ for any large $\Omega_i$. The minimal periodic orbit of type-$(\nu,0)$ for $\Phi^t_{G_{i}}$ is converted to be minimal periodic orbit of $\phi^{\tau}_{L_1}$.
As it was shown in \cite{CZ1}, the hyperbolicity of such minimal periodic orbit is uniquely determined by the nondegeneracy of the minimal point of the following function
$$
F(x_2,\Omega_i,E)=\inf_{\gamma(0)=\gamma(2\pi)=x_2}\int_{0}^{2\pi} L_1\Big(\dot\gamma(\tau),\gamma(\tau),\Omega_i\tau,E\Big)d\tau.
$$
Because of the condition $\gamma(0)=\gamma(2\pi)=x_2$, the term $\frac{A_{12}E}{A_{22}}\dot x_2$ does not contribute to $F$, so we can drop it. Let $\gamma_{\Omega_i,E}(\tau,x_2)$ be the minimizer of $F(x_2,\Omega_i,E)$, along which the action is equal to $F(x_2,\Omega_i,E)$. Then, $|\dot\gamma_{\Omega_i,E}(\tau,x_2)|$ is uniformly bounded for any large $\Omega_i$. As the system has one degree of freedom, $\frac{2\pi}{\Omega_i}$-periodical in $\tau$, the minimum of $F$ determines an $\frac{2\pi}{\Omega_i}$-periodic curve $\gamma^*_{\Omega_i,E}$. We shall see later that $|\dot\gamma^*_{\Omega_i,E}(\tau)|\to 0$ as $\Omega_i\to\infty$.

Although the Lagrangian $L_1$ depends on $\Omega_i$ in a singular way when $\Omega_i\to\infty$, the function $F$ appears regular in $\Omega_i^{-1}$ as $\Omega_i\to\infty$. To see it we decompose the action
$$
F(x_2,\Omega_i,E)=F_0(x_2,\Omega_i,E)+\frac 1{\Omega_i}F_R(x_2,\Omega_i,E)
$$
where
\begin{align*}
F_0=&\int_{0}^{2\pi}\Big(\frac {1}{2A_{22}}(\dot\gamma_{\Omega_i,E}(\tau,x_2))^2 +[V](\gamma_{\Omega_i,E}(\tau,x_2))\Big)d\tau,\\
F_R=&\int_0^{2\pi}\Omega_i(V-[V])(-\Omega_i\tau,\gamma_{\Omega_i,E}(\tau,x_2))d\tau\\ &+\int_0^{2\pi}R_L(\gamma_{\Omega_i,E}(\tau,x_2),\dot\gamma_{\Omega_i,E}(\tau,x_2),\Omega_i\tau)d\tau.
\end{align*}
\begin{lem}\label{flatlemma}
Assume the potential $V\in C^4(\mathbb{T}^2,\mathbb{R})$. Then, $F_R$ is uniformly bounded in $C^{2}$-topology as $\Omega_i\to\infty$ when $x_2$ is restricted in a small neighborhood $F^{-1}(\min F)$.
\end{lem}
\begin{proof}
As the first step, we show that $F_R$ is uniformly bounded in $C^{0}$-topology. Indeed, the first integral of $F_R$ is bounded in $C^0$-topology. We expand $V$ into a Fourier series
$$
V(-\Omega_i\tau,x_2)=[V](x_2)+\sum_{k\ne 0}V_k(x_2)e^{ik\Omega_i\tau}.
$$
With the periodic boundary condition $\gamma_{\Omega_i,E}(0,x_2)= \gamma_{\Omega_i,E}(2\pi,x_2)$, the condition that $V\in C^r$ $(r\ge 4)$ and doing integration by parts we obtain,
\begin{align*}
|\textrm{The first integrale of}\ F_R|=&\Big|\Omega_i\sum_{k\ne 0}\int_{0}^{2\pi}V_k(\gamma_{\Omega_i,E}(\tau,x_2))e^{i k\Omega_i\tau}d\tau\Big|\\
=&\Big|\sum_{i\ne 0}\frac{1}{ik}\int_{0}^{2\pi}\dot V_k(\gamma_{\Omega_i,E}(\tau,x_2)) \dot\gamma_{\Omega_i,E}(\tau,x_2) e^{ik\Omega_i\tau}d\tau\Big|\\
\le&\sum_{i\ne 0}\frac{1}{|k|}\int_{0}^{2\pi}|\dot V_k||\dot\gamma_{\Omega_i,E}|d\tau
\le B\sum_{i\ne 0}\frac{1}{|k|^r}
\end{align*}
where $B=\|V\|_{C^r}\max_{\tau}|\dot\gamma_{\Omega_i,E}(\tau,x_2)|$. As the curve $\gamma_{\Omega_i,E}$ is a minimizer, $|\dot\gamma_{\Omega_i,E}(\tau,x_2)|$ keeps uniformly bounded as $\Omega_i\to\infty$. The second integral of $F_R$ is obviously bounded in $C^0$-topology. It completes the proof of the first step.
\end{proof}
\begin{pro}\label{progamma}
Each minimizer of $F(\cdot,\Omega_i,E)$, $\gamma^*_{\Omega_i,E}(\cdot,x_2)$ approaches the constant solution in $C^1$-topology: as $\Omega_i\to\infty$ we have
$$
|\gamma^*_{\Omega_i,E}(\tau,x_2)-x_2|\to 0, \qquad |\dot\gamma^*_{\Omega_i,E}(\tau,x_2)|\to 0.
$$
\end{pro}
\begin{proof}
Without of losing generality, we assume $\min[V]=0$. If $\exists$ $d>0$ and certain $\tau_{\Omega_i}$ such that $|\dot\gamma^*_{\Omega_i,E}(\tau_{\Omega_i},x_2)|\ge d>0$ holds for any large $\Omega_i$. In this case, the action of $F_0$ along the curve $\gamma^*_{\Omega_i,E}(\tau,x_2)$ does not approaches zero as $\Omega_i\to\infty$. It is guaranteed by the property that $|\frac{d^2x_2}{d\tau^2}|$ is uniformly bounded for large $\Omega_i$. To verify it, we do direct calculation. From the Hamiltonian equation produced by (\ref{barG})
\begin{equation}\label{2dequation}
\begin{cases}
\dot x_1=1+\frac{A'_{11}}{\Omega_i^2}y_1+\frac{A'_{12}}{\Omega_i}y_2,\hskip 0.5 true cm &\dot y_1=\Omega_i\frac{\partial V}{\partial x_1},\\
\dot x_2=\frac{A'_{12}}{\Omega_i}y_1+A'_{22}y_2,  &\dot y_2=\frac{\partial V}{\partial x_2},
\end{cases}
\end{equation}
we obtain that
\begin{align*}
\frac{d^2x_2}{dx_1^2}=&\frac{d}{dt}\Big(\frac{\dot x_2}{\dot x_1}\Big)\dot x_1^{-1}
=\frac{d}{dt}\Big(\frac{\frac{A_{12}}{\Omega_i}y_1+A'_{22}y_2}{1+\frac{A'_{11}}{\Omega_i^2}y_1+\frac{A'_{12}} {\Omega_i}y_2} \Big)\dot x_1^{-1}\\
=&\frac{(A'_{12}\frac{\partial V}{\partial x_1}+A'_{22}\frac{\partial V}{\partial x_2}) -(\frac{A'_{11}}{\Omega_i}\frac{\partial V}{\partial x_1}+\frac{A'_{12}}{\Omega_i}\frac{\partial V}{\partial x_2})\frac{dx_2}{dx_1}} {(1+\frac{A'_{11}}{\Omega_i^2}y_1+\frac{A'_{12}}{\Omega_i}y_2)^2}.
\end{align*}
It shows that $|\frac{d^2x_2}{dx_1^2}|$ is bounded uniformly for any large $\Omega_i$. On the other hand, the action of $L_1$ along $x_2=x_2^*$ with $x_2^*\in[V]^{-1}(0)$ is bounded by the quantity $O(\frac 1{\Omega_i})$. It implies that $\gamma^*_{\Omega_i,E}$ is not a minimizer, the contradiction leads to the proof.
\end{proof}
\begin{proof}[Continued proof of Lemma \ref{flatlemma}]
To show the boundedness of $\partial_{x_2}^{\ell}F_R$ for $\ell=1,2$, let us study the dependence of $\gamma_{\Omega_i,E}(\tau,x_2)$ and $\dot\gamma_{\Omega_i,E}(\tau,x_2)$ on $x_2$. The Hamiltonian equation generated by (\ref{Hamoned}) is the following:
\begin{equation}\label{HamLag}
\begin{aligned}
\frac{dx_2}{d\tau}=&\Omega_i\frac{A'_{12}}{A'_{22}}\Big(1-\frac 1{\sqrt{\Delta}}\Big) +\frac{(A'^2_{12}-A'_{11}A'_{22})y_2} {A'_{22}\sqrt{\Delta}},\\
\frac{dy_2}{d\tau}=&-\frac{A'_{11}}{A'_{22}\sqrt{\Delta}}\frac{\partial V}{\partial x_2},
\end{aligned}
\end{equation}
where $\Delta=(1+\frac{A'_{12}} {\Omega_i}y_2)^2-\frac{A'_{11}}{\Omega_i^2}(A'_{22}y_2^2-2V-2E\Omega_i)$. Treating the term $\Omega_i(1-\sqrt{\Delta}^{-1})$ as a function $y_2$ and $V$, we see that it remains bounded in $C^2$-topology as $\Omega_i\to\infty$. Therefore, the right hand side of Equation (\ref{HamLag}) is smooth and bounded in $C^2$-topology for any large $\Omega_i$ and bounded $y_2$.

It has been proved in \cite{CZ1} that there is a small neighborhood of the minimal point of $F(\cdot,\Omega_i,E)$ such that the minimal curve $\gamma_{\Omega_i,E}(\cdot,x_2)$ is uniquely determined by $x_2$ if $x_2$ is in the neighborhood. It implies that boundary value problem $\{x_2(0)=x_2(2\pi)=x'_2\}$ of the equation (\ref{HamLag}) is well defined provided $x'_2$ is in the neighborhood, since Eq. \ref{HamLag} is equivalent to the Lagrange equation determined by $L_1$. Therefore, there is a smooth dependence of $y_2=y_2(x'_2)$ such that the solution of the initial value problem $\{x_2(0)=x'_2,y_2(0)=y_2(x'_2)\}$ is the same as the boundary value problem. Applying the theorem of the smooth dependence of solution of ODE on its initial value, we find the first and the second derivatives of $(\gamma_{\Omega_i,E}(\tau,x_2),\partial_{\tau x_2}L_1(\gamma_{\Omega_i,E}(\tau,x_2), \dot\gamma_{\Omega_i,E}(\tau,x_2), \tau)$ with respect to $x_2$ is smooth and bounded for any large $\Omega_i$. As $L_1$ is positive definite in $\dot x_2$, the first and the second derivatives of  $\dot\gamma_{\Omega_i,E}(\tau,x_2)$ is also bounded.

By direct calculations (doing integration by parts) we find:
\begin{equation*}\label{1daoshu}
\begin{aligned}
\frac{\partial F_R}{\partial x_2}=&-\sum_{k\ne 0}\frac{1}{ik}\int_0^{2\pi}\Big(\frac{d^2V_k}{dx_2^2}\frac{\partial\gamma_{\Omega_i,E}}{\partial x_2}\dot\gamma_{\Omega_i,E} +\frac{dV_k}{dx_2}\frac{\partial\dot\gamma_{\Omega_i,E}}{\partial x_2}\Big) e^{ik\Omega_i\tau}d\tau\\
&+\int_0^{2\pi}\Big(\frac{\partial R_L}{\partial\dot x}\frac{\partial\dot\gamma_{\Omega_i,E}}{\partial x_2} +\frac{\partial R_L}{\partial x}\frac{\partial\gamma_{\Omega_i,E}}{\partial x_2}\Big)d\tau,\\
\frac{\partial^2 F_R}{\partial x_2^2}=&-\sum_{k\ne 0}\frac{1}{ik} \int_0^{2\pi}\Big(\frac{d^3V_k}{dx_2^3}\Big(\frac{\partial\gamma_{\Omega_i,E}}{\partial x_2}\Big)^2\dot\gamma_{\Omega_i,E} +\frac{dV_k}{dx_2}\frac{\partial^2\dot\gamma_{\Omega_i,E}}{\partial x_2^2} \Big) e^{ik\Omega_i\tau}d\tau\\
&-\sum_{k\ne 0}\frac{1}{ik}\int_0^{2\pi}\frac{d^2V_k}{dx_2^2}\Big(2\frac{\partial\gamma_{\Omega_i,E}}{\partial x_2} \frac{\partial\dot\gamma_{\Omega_i,E}}{\partial x_2}+ \frac{\partial^2\gamma_{\Omega_i,E}}{\partial x_2^2}\dot\gamma_{\Omega_i,E}\Big)e^{ik\Omega_i\tau}d\tau\\
&+\int_0^{2\pi}\Big(\frac{\partial^2 R_L}{\partial\dot x^2}\Big(\frac{\partial\dot\gamma_{\Omega_i,E}}{\partial x_2} \Big)^2 +\frac{\partial^2 R_L}{\partial x^2}\Big(\frac{\partial\gamma_{\Omega_i,E}}{\partial x_2}\Big)^2\Big)d\tau\\
&+\int_0^{2\pi}\Big(\frac{\partial R_L}{\partial\dot x}\frac{\partial^2\dot\gamma_{\Omega_i,E}}{\partial x_2^2} +\frac{\partial R_L}{\partial x}\frac{\partial^2\gamma_{\Omega_i,E}}{\partial x_2^2} +2\frac{\partial^2 R_L}{\partial x\partial\dot x}\frac{\partial\dot\gamma_{\Omega_i,E}}{\partial x_2} \frac{\partial\gamma_{\Omega_i,E}}{\partial x_2}\Big)d\tau.
\end{aligned}
\end{equation*}
It follows from these Formulae that $F_R$ is bounded in $C^2$-topology if $V\in C^4$.
\end{proof}

Let us calculate the second derivative of $F_0$ with respect to $x_2$:
\begin{align*}
\frac{\partial^2F_0}{\partial x_2^2}=&\int_{0}^{2\pi}\Big (\frac {1}{A'_{22}}\Big(\frac{\partial\dot\gamma_{\Omega_i,E}}{\partial x_2}\Big)^2 +\frac{d^2}{dx_2^2}[V](\gamma^*_{\Omega_i,E}) \Big(\frac{\partial \gamma_{\Omega_i,E}}{\partial x_2}\Big)^2\Big)d\tau\\
+&\int_{0}^{2\pi}\Big (\frac {1}{A'_{22}}\dot\gamma_{\Omega_i,E}\frac{\partial^2\dot\gamma_{\Omega_i,E}}{\partial x_2^2} +\frac{d}{dx_2} [V](\gamma^*_{\Omega_i,E})\frac{\partial^2\gamma_{\Omega_i,E}}{\partial^2 x_2}\Big)d\tau.
\end{align*}
The second integral approaches zero as $\Omega_i\to\infty$ if $\gamma_{\Omega_i,E}=\gamma^*_{\Omega_i,E}$. Indeed, it follows from Proposition \ref{progamma} that $|\dot\gamma^*_{\Omega_i,E}(\tau)|\to 0$ and $|\gamma^*_{\Omega_i,E}(\tau)-x_2^*|\to 0$ as $\Omega_i\to\infty$, where $x^*_2$ is a minimal point of $[V]$. Therefore, $\frac{d[V]}{dx_2}(\gamma^*_{\Omega_i,E}(\tau))\to 0$ as $\Omega_i\to\infty$.  To estimate the first integral, we note that the minimizer $\gamma^*_{\Omega_i,E}(\tau)$ stays in a small neighborhood of the minimal point of $[V]$ provided $\Omega_i$ is sufficiently large. As $V\in\mathfrak{V}_{\infty}$, certain $d>0$ exists such that $\frac{d^2}{dx_2^2}[V](\gamma^*_{\Omega_i,E})\ge d$ holds for all $\tau\in[0,2\pi]$. The linearized variational equation of (\ref{HamLag}) with the the boundary condition $\frac{\partial \gamma_{\Omega_i,E}}{\partial x_2}(0)=\frac{\partial \gamma_{\Omega_i,E}}{\partial x_2}(2\pi)=1$ admits a unique solution
$$
\Big(\frac{\partial}{\partial x_2}\gamma_{\Omega_i,E}(\tau,x_2),\frac{\partial}{\partial x_2}\frac{\partial L_1}{\partial \dot x_2}\Big(\gamma_{\Omega_i,E}(\tau,x_2),\dot\gamma_{\Omega_i,E}(\tau,x_2),\tau\Big)\Big),
$$
and the right hand side of (\ref{HamLag}) is $C^2$-smooth and uniformly bounded for any large $\Omega_i$. Therefore, certain $T>0$ exists, uniformly lower bounded for any large $\Omega_i$, such that $\frac{\partial\gamma_{\Omega_i,E}}{\partial x_2}(\tau)>\frac 12$ for all $\tau\in[0,T]\cup[2\pi-T,2\pi]$. As the minimizer is $\frac{2\pi}{\Omega_i}$-periodic, we find $\frac{\partial\gamma_{\Omega_i,E}}{\partial x_2}(\tau)>\frac 12$ for all $\tau\in[0,2\pi]$ for large $\Omega_i$. These arguments lead to the conclusion that certain $\mu>0$ and suitably large $\Omega^*>0$ exist such that $\partial^2_{x_2}F_0(\gamma^*_{\Omega_i,E}(0),\Omega_i,E)\ge 2\mu$ if $\Omega_i\ge\Omega^*$. As the function of action $F$ is a $O(\frac 1{\Omega_i})$-perturbation of $F_0$ we have
$$
\frac{\partial^2}{\partial x_2^2}F(\gamma^*_{\Omega_i,E}(0),\Omega_i,E)\ge\mu, \qquad \Omega_i\ge\Omega^*.
$$
Let $x_2^*$ be the minimal point of $F(\cdot,\Omega_i,E)$. In this case we have
$$
F(x_2,\Omega_i,E)-F(x_2^*,\Omega_i,E)\ge\mu(x_2-x_2^*)^2,
$$
if $|x_2-x_2^*|$ is suitably small,  Let $B_E:=u^--u^+$ denote the barrier function where $u^\pm$ are the backward and forward weak KAM solutions, as it was shown in \cite{CZ1}, one has
$$
B_E(x_2)-B_E(x_2^*)\ge F(x_2,\Omega_i,E)-F(x_2^*,\Omega_i,E).
$$
As barrier function is semi-concave, there exists a number $C_L> 2\mu$ such that
$$
B_E(x_2)-B_E(x_2^*)\le C_L(x_2-x_2^*)^2.
$$
It follows that the hyperbolicity of the minimizer is not weaker than $\frac 1{\Lambda_1}= \sqrt{1-\frac{2\mu}{C_L}}$.
Let us assume the contrary, denote by $(\gamma^*_E(\tau),\dot\gamma^*_E(\tau))$ the minimal periodic orbit and denote by $(\gamma^{\pm}(\tau),\dot\gamma^{\pm}(\tau))$ the orbit such that $\gamma^-(0)=\gamma^+(0)$ and they asymptotically approaches to the orbit $(\gamma^*(\tau),\dot\gamma^*(\tau))$ as $\tau\to\pm\infty$, we then have
$$
|\gamma^*_E(\pm j)-\gamma^{\pm}(\pm j)|>\frac 1{\Lambda_1}|\gamma^*_E(\pm(j-1))-\gamma^{\pm}(\pm(j-1))|
$$
if $|\gamma^*_E(0)-\gamma^{\pm}(0)|$ is suitably small. The following computation leads to a contradiction:
\begin{align*}
C_L(\gamma^{\pm}(0)-\gamma^*_E(0))^2&\ge B_E(\gamma^{\pm}(0))-B_E(\gamma^*_E(0))\\
&\ge \sum_{j=1}^{\infty}\Big(F(\gamma^-(-j)-F(\gamma_E^*(0))\Big)+\Big(F(\gamma^+(j)-F(\gamma_E^*(0))\Big)\\
&> 2\mu\frac{(\gamma^{\pm}(0)-\gamma^*_E(0))^2}{1-\Lambda^2}=C_L(\gamma^{\pm}(0)-\gamma^*_E(0))^2.
\end{align*}
We observe a fact that such hyperbolicity holds for all $E\in[1,\bar K]$. Therefore, these minimal $(E,g)$-minimal periodic orbits make up a cylinder $\Pi_{\Omega_i, \Omega_i\bar K,g}$.

Let us go back to the coordinates before the transformation (\ref{coordinatetransformation-linear}). That the new coordinate $x_1$ goes around the circle $\mathbb{T}$ once amounts to that the old coordinate $x_1$ sweeps out an angle of $\Omega$. In the original coordinate system, we have $\frac{dx_1}{d\theta}=\Omega+O(1)$. Therefore, the hyperbolicity we obtain for $\tau=2\pi$-map is almost the same as the time $\theta=2\pi$-map determined by the Hamiltonian flow $\Phi_{G_i}^{\theta}$.

To investigate whether the tangent space of $\mathbb{T}^2$ over the cylinder admits an invariant splitting, we consider the tangent bundle of $\Pi_{\Omega_i, \Omega_i\bar K,g}^i$. The tangent space at a point $z\in\Pi_{\Omega_i,\Omega_i\bar K,g}$ is two dimensional, spanned by the a vector $v'_z$ tangent to the minimal orbit passing through this point and an orthogonal vector, denoted by $v''_z$. Because the map $\Phi_{G_{\Omega_i}}^{2\pi}$ preserves the symplectic structure, the cylinder $\Pi_{\Omega_i, \Omega_i\bar K,g}^i$ is an invariant symplectic sub-manifold made up by periodic orbits, there exists a number $\lambda\ge 1$ such that
$$
\lambda^{-1}\|v_z\|\le\|D\Phi_{G_{\Omega_i}}^{2k\pi}(z)v_z\|\le\lambda\|v_z\|
$$
holds for any $v_z\in\mathrm{Span}(v'_z,v''_z)$ and for any $k\in\mathbb{Z}$. Let $k=[\frac{\lambda}{\Lambda_1}]+1$, $\Lambda=\Lambda_1^k$, then the formula (\ref{wuqiongyuan}) holds. From Equation \ref{2dequation} we see that the number $\lambda$ is uniformly bounded for any large $\Omega_i$. This completes the proof of Theorem \ref{LmHighEnergy}.
\end{proof}

Applying the theorem of normally hyperbolic invariant manifold, we find that there exists $\epsilon_{i_0}>0$ such that for $\epsilon\le\epsilon_{i_0}$ the time-$2\pi$-map $\Phi_{G_{\epsilon,i}}^{2\pi}$ also admits a NHIC $\Pi_{0,\bar K\Omega_i,g}^{i,\epsilon}$ which is a small perturbation of $\Pi_{0,\bar K\Omega_i,g}^i$.

\section{Uniform hyperbolicity}
\setcounter{equation}{0}
Recall that along the resonant path $\Gamma^*\cap\{|p-p''|\le\epsilon^{\sigma}\}$ we choose points $\{p'_i\}$ so that $p'_0=p''$, $\partial_1 h(p'_i)=Ki\sqrt{\epsilon}$, where $K\in\mathbb{Z}^+$. Around the $K_1\sqrt{\epsilon}$-neighborhood of $p'_i$ ($K_1$ is independent of $\epsilon$), the Hamiltonian is resaled to the form of (\ref{2dHmil}):
$$
G_{\epsilon,i}(x,y,\theta)=\Omega_iy_1+\frac 12\langle Ay,y\rangle -V(x)+ \epsilon^{\sigma}R_{\epsilon,i}(x,y,\theta).
$$
Let us fix a potential $V\in\mathfrak{V}_{\infty}$. By the study in the last section, there is $i_0$ independent of $\epsilon$ such that, for each $i\ge i_0$, $\Phi_{G_{i}}^{\theta}$ admits a unique invariant cylinder $\Pi_{\Omega_i,\bar K\Omega_i,g}$.

We claim that for all $i\ge i_0$, each of these cylinders is just a part of large cylinder. To verify it, we consider two adjacent subscripts $i,i+1$, and denote by $(x_j,y_j,I_j,\theta_j)$ the coordinates used for $\tilde G_{\epsilon,j}$ with $j=i,i+1$. Because of the translation (\ref{energylevel}), we find
$$
\Big(y_i-y_{i+1},\frac{\sqrt{\epsilon}}{\omega_{3,i}}I_i-\frac{\sqrt{\epsilon}}{\omega_{3,i+1}}I_{i+1}\Big)=\frac 1{\sqrt{\epsilon}}(p'_{i}-p'_{i+1}).
$$
As $I_j=-G_{\epsilon,j}$ solves Equation \ref{reduc}, the energy levels of $G_{\epsilon,i}$ and $G_{\epsilon,i+1}$ match in the following way:
\begin{equation}\label{match}
G_{\epsilon,i}^{-1}\Big(\frac{\omega_{3,i}}{\omega_{3,i+1}}E+\frac {\omega_{3,i}}{\epsilon}(p'_{3,i+1}-p'_{3,i})\Big)=G_{\epsilon,i+1}^{-1}(E),
\end{equation}
where we use the notation $p'_i=(p'_{1,i},p'_{2,i},p'_{3,i})$ and so on.

By the choice of $p'_i$ we have following identities
$$
h(p'_i)=h(p'_{i+1}), \qquad \partial_2 h(p'_i)=\partial_2h(p'_{i+1})=0, \qquad \partial_1h(p'_i)=Ki\sqrt{\epsilon}.
$$
From these identities we obtain
\begin{equation}\label{estimate}
\begin{aligned}
\Big\langle\frac{\partial^2h}{\partial p^2}(p'_i),p'_{i+1}-p'_i\Big\rangle+O(\|p'_{i+1}-p'_i\|^2) =&(K\sqrt{\epsilon}, 0,\omega_{3,i+1}-\omega_{3,i})^t,\\
Ki\sqrt{\epsilon}(p'_{1,i+1}-p'_{1,i})+\omega_{3,i}(p'_{3,i+1}-p'_{3,i})=&O(\|p'_{i+1}-p'_i\|^2).
\end{aligned}
\end{equation}
It follows from the conditions $Ki\sqrt{\epsilon}\le O(\epsilon^{\sigma})$, $\|p'_{i+1}-p'_i\|\ll 1$ and the second formula of (\ref{estimate}) that $|p'_{3,i+1}-p'_{3,i}|\ll \|p'_{i+1}-p'_i\|$. It is also impossible that $|p'_{2,i+1}-p'_{2,i}|\gg\max\{|p'_{1,i+1}-p'_{1,i}|, |p'_{3,i+1}-p'_{3,i}|\}$ since the matrix $\frac{\partial^2h}{\partial p^2}(p'_i)$ is positive definite, it can not map a vector to someone almost orthogonal to itself, as shown in the first formula of (\ref{estimate}). Therefore, there exist positive constants $K_2, K_3>0$ independent of $\epsilon$ such that
$$
\|p'_{i+1}-p'_i\|\le K_2\sqrt{\epsilon}
$$
and
$$
|p'_{3,i+1}-p'_{3,i}|\le \frac 1{\omega_{3,i}}\Big|Ki\sqrt{\epsilon}(p'_{1,i+1}-p'_{1,i})+O(\|p'_{i+1}-p'_i\|^2)\Big|\le \frac{K_3}{\omega_{3,i}}i\epsilon.
$$
Recall the coordinate rescaling of (\ref{energylevel}). Since $\frac{\omega_{3,i}}{\omega_{3,i+1}}$ is close to 1, for $\bar K>3\frac{K_2}{\omega_{3,i}K}$ the energy level $G_{\epsilon,i}^{-1}(\bar K\Omega_i)$ is contained in the set where $G_{\epsilon,i+1}>\Omega_i$. This $\bar K$ is available if we choose suitable $K_1>0$.

Obviously, in the region where both $G_{\epsilon,i}$ and $G_{\epsilon,i+1}$ remains valid, there is a cylinder $\Pi_{0,\bar K\Omega_i,g}^{i,\epsilon}\cap\Pi_{0,\bar K\Omega_{i+1},g}^{i+1,\epsilon}$ as well.
It follows from the uniqueness of such invariant cylinder. Therefore, we get an NHIC which extends from the energy level $G_{\epsilon}^{-1}(E_{i_0}=\omega_3p'_{3,{i_0}}\epsilon^{-1})$ to the energy level $G_{\epsilon}^{-1}(E_{i_1}=\omega_3p'_{3,{i_1}}\epsilon^{-1})$ such that $|p''-p'_{i_1}|=D\epsilon^{\sigma-\frac 12}$. The tangent space of $\mathbb{T}^2$ over the whole $\Pi_{E_{i_0},E_{i_1},g}$ admits normally hyperbolic invariant splitting of (\ref{wuqiongyuan}). For $i<0$, the situation is the same, instead of considering the class $g$, we consider the class $-g$.

Back to the original coordinates, for the class $g$ as well as for $-g$, there is a NHIC which extends from $\Omega_{i_0}\sqrt{\epsilon}$-neighborhood of the double resonant point $p''$ to the border of the disk $\{|p-p''|\le D\epsilon^{\sigma}\}$.

Since $\|p'_{i+1}-p'_i\|\le K_2\sqrt{\epsilon}$, there are as many as $O([\epsilon^{\sigma-\frac 12}])$ points $\{p'_i\}$ along the resonant path $\Gamma'$. Around at most $2i_0+1$ points (the number is independent of $\epsilon$), the situation need to be handled in the way treated in \cite{CZ1}. For each of them, there is an open-dense set $\mathfrak{V}_i\subset C^r(\mathbb{T}^2,\mathbb{R})$. For each $V\in\mathfrak{V}_i$, the Hamiltonian flow $\Phi_{G_i}^{\theta}$ admits NHICs in the domain with certain normal hyperbolicity independent of $\epsilon$. Therefore, certain $\epsilon_i>0$ exists such that for each $\epsilon\le\epsilon_i$, the cylinders survive the time-periodic perturbation $\Phi_{G_i}^{\theta}\to\Phi_{G_ {\epsilon,i}}^{\theta}$. Note the Hamiltonian $G_{\epsilon,i}$ is a local expression of $G_{\epsilon}$.

Now the situation becomes clear. One cylinder extends from $G_{\epsilon}^{-1}(\epsilon^d)$ to $G_{\epsilon}^{-1}(E_0)$, another cylinder extends from $G_{\epsilon}^{-1}(E_{i_0})$ to $G_{\epsilon}^{-1}(D\epsilon^{\sigma-\frac 12})$. Between the energy level $G_{\epsilon}^{-1}(E_0)$ and $G_{\epsilon}^{-1}(E_{i_0})$ there are finitely many pieces of NHICs. Each energy level intersects these NHIC's along one or two circles. Let
$$
\mathfrak{V}=\Big(\bigcap_{|j|<i_0}\mathfrak{V}_j\Big)\cap\mathfrak{V}_{\infty},
$$
for each $V\in\mathfrak{V}$ we choose $\epsilon_V=\min\{\epsilon_0,\cdots,\epsilon_{\pm i_0},\}$. The first part of Theorem \ref{mainresult} is proved for $V\in\mathfrak{V}$ and $\epsilon\le\epsilon_V$.

\section{Aubry sets along resonant path: near double resonance}
\setcounter{equation}{0}
Since the NHICs obtained are overflowing, we need to identify whether the Aubry sets along resonant path remain in the cylinder.

An irreducible class $g\in H_1(\mathbb{T}^2,\mathbb{R})$ determines a channel of first cohomology classes
$$
\mathbb{C}_{g,G}=\bigcup_{\nu\in\mathbb{R}^+}\mathscr{L}_{\beta_{G}}(\nu g),\qquad \mathbb{C}_{E',E'_1,g,G}=\Big\{c\in\mathbb{C}_{g,G}: \alpha_{G}(c)\in[E',E'_1]\Big\}.
$$
\begin{theo}\label{AubrySet}
For the Hamiltonian $G_{\epsilon,0}$ of $($\ref{Hamilton2}$)$, a class $g\in H_1(\mathbb{T}^2,\mathbb{Z})$ and a large positive number $E'_1>0$, there exists a residual set $\mathfrak{V}\subset C^{r}(\mathbb{T}^2,\mathbb{R})$ $(r\ge 5)$. For each $V\in\mathfrak{V}$ there are numbers $\epsilon_0>0$ and $d>0$ such that for $\epsilon\le\epsilon_0$ it holds for each $c\in \mathbb{C}_{2\epsilon^d,E'_1,g}$ that the Aubry set $\tilde{\mathcal{A}}(c)$ of $G_{\epsilon,0}$ lies on invariant cylinder.
\end{theo}
To prove this theorem we need some preliminary works.
\begin{lem}\label{lem2.4.}
For the Hamiltonian $G_{\epsilon,0}$ of $($\ref{Hamilton2}$)$, if an orbit $z(s)$ remains in a bounded region $\Omega\subset T^*M$ for $s\in[s_0,s_1]$, some constant $K>0$ exists, independent of $\epsilon$ such that the variation of energy along the orbit $z(s)$ is bounded by
$$
|G_{\epsilon,0}(z(s_1),s_1)-G_{\epsilon,0}(z(s_0),s_0)|\le K|s_1-s_0+1|\epsilon^{\sigma}.
$$
\end{lem}
\begin{proof} Along an orbit $z(s)$ the variation of the energy is given by
\begin{equation}\label{derivativeoftime}
\frac d{d\theta}G_{\epsilon,0}(z(\theta),\theta)=\frac{\partial}{\partial\theta}G_{\epsilon,0}(z(\theta),\theta) =\omega_3\epsilon^{\sigma- \frac 12}\frac{\partial R_{\epsilon,0}}{\partial\tau}
\end{equation}
where $\tau=\omega_3\frac{\theta}{\sqrt{\epsilon}}$. Recall $R_{\epsilon,0}$ is regular and $2\pi$-periodic in $\tau$, see (\ref{reduction}), we expend $R$ into Fourier series
$$
\partial_{\theta}G_{\epsilon,0}(z,\theta)=\omega_3\epsilon^{\sigma-\frac 12}\sum_{k\neq 0}R_k(z)e^{ik\frac{\omega_3\theta}{\sqrt{\epsilon}}}.
$$
Integrating by parts, we have
\begin{equation}\label{Riemann-Lebesgue}
\begin{aligned}
\epsilon^{\sigma-\frac 12}\int_{s_0}^{s_1}R_k(z(\theta))e^{ik\frac{\omega_3\theta}{\sqrt{\epsilon}}}d\theta&= \frac{\epsilon^{\sigma}}{i\omega_3k}R_k(z(\theta))e^{ik\frac{\omega_3\theta}{\sqrt{\epsilon}}}\Big|_{s_0}^{s_1}\\
&-\frac{\epsilon^{\sigma}}{i\omega_3k}\int_{s_0}^{s_1}\langle\partial R_k,\dot z(\theta)\rangle e^{ik\frac{\omega_3\theta}{\sqrt{\epsilon}}}d\theta.
\end{aligned}
\end{equation}
As the perturbation term $R$ is $C^4$-smooth and $2\pi$-periodic in $\tau=\omega_3\frac{\theta}{\sqrt{\epsilon}}$, so we have
$$
|R_k|\le\frac {\|R\|_{C^3}}{2\pi|k|^3},\qquad |\partial R_k|\le \frac {\|\partial_zR_{\epsilon,0}\|_{C^3}}{2\pi|k|^3}.
$$
So, by setting
$$
K=\frac {1}{\omega_3\pi}\max_{z\in\Omega}\Big\{\|R_{\epsilon,0}\|_{C^3},\Big(|Ay|+\Big|\frac{\partial V}{\partial x}\Big|\Big)\|\partial_zR_{\epsilon,0}\|_{C^3}\Big\}\sum_{k\neq 0}\frac{1}{|k|^4}.
$$
which is independent of $\epsilon$, it follows from (\ref{derivativeoftime}) and (\ref{Riemann-Lebesgue}) that,
\begin{align*}
|G_{\epsilon,0}(z(s_1),s_1)-G_{\epsilon,0}(z(s_0),s_0)|\le&\Big|\sum_{k\neq 0}\frac{\epsilon^{\sigma}}{i\omega_3k}R_k(z(\theta))e^{ik\frac{\omega_3\theta}{\sqrt{\epsilon}}}\Big|_{s_0}^{s_1}\\
&- \sum_{k\neq 0}\frac{\epsilon^{\sigma}}{i\omega_3k}\int_{s_0}^{s_1}\langle\partial R_k,\dot z(\theta)\rangle e^{ik\frac{\omega_3\theta}{\sqrt{\epsilon}}}d\theta\Big|
\end{align*}
the right hand side is not bigger than $K(s_1-s_0+1)\epsilon^{\sigma}$.
\end{proof}
Since $\Pi_{E_i,E_{i+1},g}$ is a hyperbolic cylinder, the channel $\mathbb{C}_{E_i,E_{i+1},g,G_0}$ admits a foliation of sections of line (one-dimensional flat), denoted by $\{I_E\}$. Restricted on each $I_E$ $\alpha_{G_0}$ keeps constant, while restricted on a line $\Gamma_g$ orthogonal to these flats, the function is smooth because $G_0$ can be treated as a Hamiltonian with one degree of freedom when it is restricted on the cylinder. Therefore, the function $\alpha_{G_0}$ is smooth in $\mathbb{C}_{0,E_1,g,G_0}$ and $\mathbb{C}_{E_i,E_{i+1},g,G_0}$.
\begin{pro}\label{pro1}
There exists a number $N>1$, independent of $\epsilon$, so that the Aubry set for $c\in\Gamma_g\cap\alpha^{-1}_{G_{\epsilon,0}}(E)$ with $E\ge N\epsilon^{d}$ lies in the NHIC, each orbit in this set does not hit the energy level set $G_{\epsilon,0}^{-1}(E)$ with $E\le\epsilon^d$.
\end{pro}
\begin{proof}
We only need to prove the conclusion for the Hamiltonian $G'_{\epsilon}$ of (\ref{modification}), because $G'_{\epsilon}=G_{\epsilon,0}$ when it is restricted on the set where $G'_{\epsilon}\in[\epsilon^d,E_1]$. So, each orbit in the Aubry set lies in the cylinder forever. If the proposition does not hold, there would exist an orbit $z(s)$ in the Aubry set for $c\in\Gamma_g\cap\alpha^{-1}_{G'_{\epsilon}}(N\epsilon^{d})$, which hits the energy level $G'^{-1}_{\epsilon}(\epsilon^d)$ at the time $s=s_0\mod\sqrt{\epsilon}$, i.e. $G'_{\epsilon}(z(s_0),s_0)=\epsilon^d$. Due to Lemma \ref{lem2.4.}, it returns to a neighborhood of $z(s_0)$ after a time $S=O(|\ln\epsilon^d|)$ (cf. formula (\ref{period})) and
\begin{equation}\label{cylindereq19}
|G'_{\epsilon}(z(S+s_0),S+s_0)-G'_{\epsilon}(z(s_0),s_0)|\le K(S+1)\epsilon^{\sigma}.
\end{equation}
As $G_0^{-1}(E)\cap\Pi_{E_d,E_{1},g}$ is an invariant circle for $\Phi^t_{G_0}$, the perturbed cylinder is $O(\epsilon^{\sigma})$-close to the original one \cite{BLZ} and the cylinder may be crumpled but at most up to the order $O(E^{-2\mu_6})$ (cf. (\ref{cylindereq9})), so there is $S=O(|\ln\epsilon^d|)$ such that
$$
\|z(S+s_0)-z(s_0))\|\le C_{13}(S+1)\epsilon^{\sigma-\mu_6}.
$$
Since $z(s)$ is in the Aubry set for the first cohomology class $c$, the curve $x(s)$ is $c$-static. Using $\alpha_{G'_{\epsilon}}$ and $\alpha_{G_0}$ to denote the $\alpha$-function for $G'_{\epsilon}$ and $G_0$ respectively, it follows that
\begin{equation}\label{cylindereq20}
\Big|\int_{s_0}^{S+s_0}(L_{G'_{\epsilon}}(x(s),\dot x(s),s)-\langle c,\dot x(s)\rangle+\alpha_{G'_{\epsilon}}(c))ds \Big|\le C_{14}(S+1)\epsilon^{\sigma-\mu_6}.
\end{equation}
Since the cylinder $\Pi_{E_d,E_{1},g}\times\sqrt{\epsilon}\mathbb{T}$ is $\epsilon^{\sigma}$-close to $\tilde\Pi_{E_d,E_{1},g}$, $\exists$ a $c'$-minimal orbit $z'(s)$ of $\Phi^s_{G_0}$ on $\Pi_{E_d,E_{1},g}$ such that $\alpha_{G_0}(c')=\epsilon^d$ and $\|z'(s_0)-z(s_0)\|\le O(\epsilon^{\sigma-\mu_6})$. Let $\Gamma_x=\bigcup_{s=s_0}^{s_0+S}(x(s),y(s))$ and $\Gamma_{x'}=\bigcup_{s=s_0}^{s_0+S'}(x'(s),y'(s))$ where $S'$ is the period of $x'(s)$, we have an estimate on the Hausdorff distance
$d_H(\Gamma_x,\Gamma_{x'})\le O((S+1)\epsilon^{\sigma-\mu_6})$. So,
$$
\int_{\Gamma_x}\langle y,dx\rangle-\int_{\Gamma_{x'}}\langle y,dx\rangle=O((S+1)\epsilon^{\sigma-\mu_6}).
$$
Because of $G_0(x'(s),y'(s))\equiv\alpha_{G}(c')$ we have
\begin{align}
0&=\int_0^{S'}(L_{G_0}(x'(t),\dot x'(t))-\langle c',\dot x'(t)\rangle+\alpha_{G_0}(c'))dt \notag\\
&=\int_0^{S'}\langle y'(s)-c',\dot x'(s)\rangle ds\notag
\end{align}
Let $\bar x(s)$ be the lift of $x(s)$ to the universal covering space, it follows that
\begin{equation}\label{guodu}
\begin{aligned}
&\int_{s_0}^{S+s_0}\langle y(s)-c,\dot x(s)\rangle ds\\
=&\int_{s_0}^{S+s_0}\langle y(s)-c',\dot x(s)\rangle ds-\int_{s_0}^{S'+s_0}\langle y'(s)-c',\dot x'(s)\rangle ds\\
&-\langle c-c',\bar x(S+s_0)-\bar x(s_0)\rangle\\
=&\int_{\Gamma_x}\langle y,dx\rangle-\int_{\Gamma_{x'}}\langle y',dx'\rangle+O((S+1)\epsilon^{\sigma-\mu_6})\\
&-\langle c-c',\bar x(S+s_0)-\bar x(s_0)\rangle\\
=&-\langle c-c',\bar x(S+s_0)-\bar x(s_0)\rangle+O((S+1)\epsilon^{\sigma-\mu_6})
\end{aligned}
\end{equation}
As $G'_{\epsilon}(z(s_0),s_0)=\alpha_{G_0}(c')$, it follows from (\ref{cylindereq19}) that, for all $s\in[s_0,S+s_0]$, we have
$$
\alpha_{G'_{\epsilon}}(c)-G'_{\epsilon}(x(s),y(s),s)\ge\alpha_{G'_{\epsilon}}(c)-\alpha_{G_0}(c')-O((S+1) \epsilon^{\sigma-\mu_6}).
$$
Consequently, by using the formulae (\ref{cylindereq19}) and (\ref{guodu}) we have
\begin{equation}\label{cylindereq21}
\begin{aligned}
&\int_{s_0}^{S+s_0}(L_{G'_{\epsilon}}(x(s),\dot x(s),s)-\langle c,\dot x(s)\rangle+\alpha_{G'_{\epsilon}} (c))ds\\
=&\int_{s_0}^{S+s_0}\Big(\langle y(s)-c,\dot x(s)\rangle+(\alpha_{G'_{\epsilon}}(c)- G'_{\epsilon}(x(s),y(s),s))\Big)ds\\
\ge&\, (\alpha_{G'_{\epsilon}}(c)-\alpha_{G_0}(c'))S-\langle c-c',\bar x(S+s_0)-\bar x(s_0)\rangle-O((S+1)\epsilon^{\sigma}).
\end{aligned}
\end{equation}

To derive contradiction between the right-hand-side of above inequality and (\ref{cylindereq20}), we note that the function $\alpha_{G_0}$ keeps constant along each flat in the channel $\mathbb{C}_{0,E_1,g,G_0}$. The frequency vector $\omega(c)$ is therefore parallel to the direction of $\Gamma_g$. To get the norm of $\omega(c)$, we assume the general case $g=k_1g_1+k_{2}g_{2}$ and consider the Hamiltonian in the finite covering space $\bar M=\bar k_1\mathbb{T}\times\bar k_2\mathbb{T}$ where $\bar k_m=k_1g_{1m}+k_{2}g_{2m}$ for $m=1,2$ if we write $g_j=(g_{j1},g_{j2})$ for $j=1,2$. In the space $T^*\bar M$ there are $k_1+k_{2}$ fixed points for the return map. According to Formula (\ref{regularenergyeq2}), for small $E>0$ the period of the frequency $\nu g$ is $T_{\nu g}=\lambda_1^{-1}(k_1+k_2)(-\ln E+\tau_g(E))$ where $\tau_g(E)$ is uniformly bounded as $E\to 0$. Since $\partial\alpha_{G_0}=\omega=\nu g$, one has
\begin{equation}\label{cylindereq15}
\frac{\lambda_1}{(|\ln E|+\tau_g(E))(k_1+k_2)}=|\omega|,\qquad \forall\ c\in \Gamma_g.
\end{equation}
Let $c^*$ be the class such that $\alpha_{G_0}(c^*)=\alpha_{G'_{\epsilon}}(c)$, then $\alpha_{G_0}(c^*)-\alpha_{G_0}(c')=\epsilon^d>0$. Since $\alpha_{G_0}$ is convex, $\alpha_{G_0}(c^*)>\alpha_{G_0}(c')$,
$$
\langle c^*-c',\partial\alpha_{G_0}(c^*)\rangle>\alpha_{G_0}(c^*)-\alpha_{G_0}(c').
$$
As $c^*,c'\in\Gamma_g$, $c^*-c'$ is parallel to $\partial\alpha_{G_0}(c^*)$. It follows from (\ref{cylindereq15}) that some constant $C_{15}>0$ exists such that
\begin{equation}\label{difference}
|c^*-c'|\ge \frac 1{\|\partial\alpha_{G_0}(c^*)\|}\Big(\alpha_{G_0}(c^*)-\alpha_{G_0}(c')\Big)\ge C_{15}\epsilon^{d}|\ln\epsilon^{d}|.
\end{equation}
To measure the distance between $c^*$ and $c$, we exploit the convexity of the $\alpha$-function and get $|\langle c-c^*,\omega(c^*)\rangle|\le|\alpha_{G_0}(c^*)-\alpha_{G_0}(c)|=|\alpha_{G_{\epsilon,0}}(c)-\alpha_{G_0}(c)| =\epsilon^{\sigma}$. It follows from that the $\alpha$-function undergoes small variation: $|\alpha_L(c)-\alpha_{L'}(c)|\le\varepsilon$ for small perturbation $L'\to L$ with $\|L'-L\|_{C^1}\le\varepsilon$ (\cite{C11}). Therefore, we get
\begin{equation}\label{cylindereq18}
|c^*-c|\le C_{16}\epsilon^{\sigma}|\ln\epsilon^{d}|.
\end{equation}

Be aware that $\alpha_{G_0}$ is smooth and strictly convex when it is restricted on the line $\Gamma_g$ and the first cohomology classes $c',c^*\in\Gamma_g$ are uniquely determined so that $\alpha_0(c')=\epsilon^d$, $\alpha_0(c^*)=N\epsilon^d$ where the number $N$ is chosen such that
$$
\ln(N-1)=3\max\{|\sup\tau_g(E)|,1\}.
$$
Let $c''\in\Gamma_g$ such that $\alpha_0(c'')=(N-1)\epsilon^d$, we find that
\begin{equation*}
\alpha_{G_0}(c^*)-\alpha_{G_0}(c'')>\langle\omega'',c^*-c''\rangle, \qquad \alpha_{G_0}(c'')-\alpha_{G_0}(c')>\langle\omega',c''-c'\rangle,
\end{equation*}
where $\omega'=\partial\alpha_0(c')$ and $\omega''=\partial\alpha_0(c'')$. It follows that
\begin{equation}\label{cylindereq17}
\alpha_{G_0}(c^*)-\alpha_{G_0}(c')>\langle c^*-c',\omega'\rangle+\langle c^*-c'',\omega''-\omega'\rangle.
\end{equation}

In the way to get (\ref{difference}) one finds that
\begin{equation}\label{cylindereq16}
|c^*-c''|\ge C_{17}\epsilon^{d}|\ln\epsilon^{d}|,
\end{equation}
where the number $C_{17}$ depends on $N$. One also has
\begin{equation}\label{cylindereq19}
|\omega''-\omega'|\ge\frac{\lambda_1\max\{|\sup\tau_g(E)|,1\}}{(k_1+k_2)(|\ln\epsilon^d|+\tau_g(\epsilon^d))(|\ln\epsilon^d|-\ln(N-1)+ \tau_g((N-1)\epsilon^d))}.
\end{equation}
Since $\langle\omega''-\omega',c^*-c''\rangle=|\omega''-\omega'||c^*-c''|$ (restricted on the cylinder, the system has only one degree of freedom, so they are treated as scalers, not vectors), one obtains from (\ref{difference}), (\ref{cylindereq18}), (\ref{cylindereq17}), (\ref{cylindereq16}) and (\ref{cylindereq19}) that
$$
\alpha_{G_0}(c)-\alpha_{G_0}(c')-\langle c-c',\omega'\rangle\ge C_{18}\frac{\epsilon^{d}}{|\ln\epsilon^{d}|},
$$
from which we see that the right hand side of (\ref{cylindereq21}) is lower bounded by $C_{19}\epsilon^{d}$ where $C_{19}>0$ is a constant. Because $\mu_6$ is very small, the formula (\ref{cylindereq21}) contradicts (\ref{cylindereq20}) provided $\sigma>d+\mu_6$. It completes the proof.
\end{proof}
\begin{pro}\label{Aubrysetplace2}
If the Aubry set for $c\in\Gamma_g\cap\alpha^{-1}_{G_{\epsilon,0}}(E)$ is contained in the NHIC, and $E>0$ is independent of $\epsilon$, each orbit in this set does not hit the energy level set $G_{\epsilon,0}^{-1}(E\pm \epsilon^{\frac 13\sigma})$ if $\epsilon>0$ is sufficiently small.
\end{pro}
\begin{proof}
If an orbit $z(s)$ of the Aubry set hits the energy level $G_{\epsilon,0}^{-1}(E- \epsilon^{\frac 13\sigma})$, following the proof of \ref{pro1} we also have (\ref{cylindereq20}) and (\ref{cylindereq21}). Again, we are going to show the contradiction between them.

Let $c^*$ be the class such that $\alpha_{G_0}(c^*)=\alpha_{G_{\epsilon,0}}(c)$, then $\alpha_{G_0}(c^*)-\alpha_{G_0}(c')=\epsilon^{\frac 13\sigma}$. Let $c'\in\Gamma_g$ such that $\alpha_{G_0}(c')=E-\epsilon^{\frac 13\sigma}$. Similar to way to get (\ref{difference}), note the period is of order one, we obtain
\begin{equation}\label{difference1}
|c^*-c'|\ge C_{20}\epsilon^{\frac13\sigma}.
\end{equation}
As $\alpha_{G_0}$ is strictly convex, one obtains from (\ref{difference1}), (\ref{cylindereq18}) and (\ref{cylindereq17}) that
$$
\alpha_{G_0}(c)-\alpha_{G_0}(c')-\langle c-c',\omega'\rangle=\frac 12|\partial^2 \alpha_{G_0}(\nu c+(1-\nu)c')||c'-c|^2\ge C_{21}\epsilon^{\frac23\sigma},
$$
from which we see that the right hand side of (\ref{cylindereq21}) is lower bounded by $O(\epsilon^{\frac 23\sigma})$. As $\mu_6$ is very small, the formula (\ref{cylindereq21}) contradicts (\ref{cylindereq20}) provided $\sigma>d+\mu_6$. The proof for $E+\epsilon^{\frac 13\sigma}$ appears the same.
\end{proof}

\begin{proof}[Proof of Theorem \ref{AubrySet}]
For the Hamiltonian $G_0$ with $V\in\mathfrak{V}$, there are at most finitely bifurcation points $0<E_1,E_2,\cdots E_k\le E'_1$. The Aurby set $\tilde{\mathcal{A}}(c)$ for $G_0$ is a $(E,g)$-minimal orbit if $c\in\mathscr{L}_{\beta_{G_0}}(\nu g)$ and $\alpha_{G_0}(c)\neq E_i$ for $i=1,2,\cdots k$. At those bifurcation points the Aubry set consists of exactly two $(E,g)$-minimal periodic orbits. Such periodic orbits make up several pieces of NHICs which admit a continuation to the energy level of $E_i\pm\Delta$, denoted by $\Pi_{0,E_1+\Delta,g}$ and $\Pi_{E_i-\Delta,E_{i+1}+\Delta,g}$ respectively. The continuation is made up by local $(E,g)$-minimal periodic orbits. Restricted on the cylinder $\Pi_{E_i-\Delta,E_{i+1}+\Delta,g}$, the Hamiltonian has one degree of freedom, associated with a smooth $\alpha$-function denoted by $\alpha_i$: $c_1\in [c_1^i-\Delta c_1,c_1^{i+1}+\Delta c_1]\to\mathbb{R}$. The first cohomology class $c_1$ determines uniquely a flat $I_E\subset \mathbb{C}_{E_i,E_{i+1},g,G_0}$ such that $\alpha_i(c_1)=\alpha_{G_0}(I_E)$ if $c_1\in[c_1^i,c_1^{i+1}]$. Indeed, one has $\alpha_{G_0}(c_1,c_2)=\alpha_i(c_1)$ if $(c_1,c_2)\in I_{\alpha_{G_0}(c_1,c_2)}$ and we use certain coordinates so that $g=(1,0)$. By the definition, we have $\alpha_{i-1}(c_1^{i})=\alpha_i(c_1^{i})$, $\alpha_{i-1}(c_1)\ge\alpha_i(c_1)$ for $c_1\in[c_1^i,c_1^i+\Delta c_1]$ and $\alpha_{i}(c_1)\le\alpha_{i+1}(c_1)$ for $c_1\in[c_1^{i+1}-\Delta c_1,c_1^{i+1}]$. It follows from the generic condition ({\bf H3}) that
$$
\frac d{dc_1}\alpha_{i-1}(c_1^i)>\frac d{dc_1}\alpha_i(c_1^i), \qquad \forall\ i.
$$
Under the small perturbation $\epsilon^{\sigma}R$, large part of NHICs survive, such as $\tilde\Pi_{\epsilon^d,E_1+\Delta-\epsilon^d}$ and $\tilde\Pi_{E_i-\Delta+\epsilon^d,E_{i+1}+\Delta-\epsilon^d}$. The former is $\epsilon^{\sigma}$-close to $\Pi_{\epsilon^d,E_1+\Delta-\epsilon^d,g} \times\sqrt{\epsilon} \mathbb{T}$, the latter is $\epsilon^{\sigma}$-close to $\Pi_{E_i-\Delta+\epsilon^d,E_{i+1}+\Delta-\epsilon^d,g} \times\sqrt{\epsilon}\mathbb{T}$.

For the Hamiltonian $G_0$ and the class $c^i\in\mathbb{C}_{g,G_0}$ with $\alpha_{G_0}(c^i)=E_i$, the Aubry set consists of two $\nu_ig$-minimal periodic orbits, the Ma\~n\'e set contains these two periodic orbits plus some orbits connecting them (hetroclinic orbits). For the Hamiltonian $G_{\epsilon,0}$ and the class $c\in\mathbb{C}_{\epsilon^d,E'_1,g}$ so that $|\alpha_{G_{\epsilon,0}}(c)-E_i|\le\epsilon^{\sigma}$, the Ma\~n\'e set $\tilde{\mathcal{N}}(c)$ stays in a small neighborhood of cylinders $\tilde\Pi_{E_{i-1}-\Delta+\epsilon^d,E_{i}+\Delta -\epsilon^d,g}$ and $\tilde\Pi_{E_{i}-\Delta+\epsilon^d,E_{i+1}+\Delta-\epsilon^d,g}$. It is due to the upper semi-continuity of Ma\~n\'e on small perturbations. So, it follows from the hyperbolic structure that each ergodic minimal measure for this class has its support in the cylinder either $\tilde\Pi_{E_{i-1}-\Delta+\epsilon^d,E_{i}+\Delta -\epsilon^d,g}$ or $\tilde\Pi_{E_{i}-\Delta+\epsilon^d,E_{i+1}+\Delta-\epsilon^d,g}$.

Since the energy level set $G_{\epsilon,0}^{-1}(E)$ is in $\epsilon^{\sigma}$-neighborhood of $G_0^{-1}(E)$, we obtain from Proposition \ref{Aubrysetplace2} and the condition ({\bf H3}) that for $c\in\mathbb{C}_{g,G_{\epsilon,0}}$ such that $\alpha_{G_{\epsilon,0}}(c)$ is close to $E_i$ we have
$$
\tilde{\mathcal{A}}(c)\subset \tilde\Pi_{E_{i-1},E_i+K\epsilon^{\frac \sigma 3},g}\cup\tilde\Pi_{E_{i}-K\epsilon^{\frac \sigma 3},E_{i+1},g}
$$
where $K\ge 2\max\{(\frac d{dc_1}\alpha_{i-1}(c_1^i)-\frac d{dc_1}\alpha_i(c_1^i))^{-1},1\}\|R\|_{\infty}$. Since $\Delta>0$ is independent of $\epsilon$, the Aubry set completely lies on the cylinders if $\epsilon>0$ is suitably small.

To verify that the Aubry set $\tilde{\mathcal{A}}(c)$ with $\alpha_{G_{\epsilon,0}}(c)=2\epsilon^d$ is contained in the cylinder, we apply Theorem \ref{cylinderforepsilon}. The cylinder $\tilde\Pi_{\frac 12\epsilon^d,E_1+\Delta-\epsilon^d,g}$ lies in $O(\epsilon^{\sigma})$-neighborhood of the cylinder $\Pi_{\frac 12\epsilon^d,E_1+\Delta-\epsilon^d,g}\times\sqrt{\epsilon}\mathbb{T}$. By the choice of the number $d$ in (\ref{d>0}), we see that $G_{\epsilon,0}(z,\theta)\ge\epsilon^d$ if $(z,\theta)\in\tilde\Pi_{\frac 12\epsilon^d,E_1+\Delta-\epsilon^d,g}$ provided $\epsilon>0$ is sufficiently small. Applying Lemma \ref{pro1}, we then complete the proof.
\end{proof}

This section also completes the proof of the second part of Theorem \ref{mainresult}.

\section{Criterion for strong and weak double resonance}
\setcounter{equation}{0}
Given a perturbation $\epsilon P(p,q)$, it is natural to ask, along the resonant path $\Gamma'$, how many double resonances need to be treated as strong double resonance. What we write below is in fact the proof of Theorem \ref{consequence}.

\begin{proof}[The proof of Theorem \ref{consequence}] On the path $\Gamma'$ the resonance condition $\langle\partial h(p),k'\rangle=0$ is always satisfied and at each double resonant point some other $k''\in\mathbb{Z}^3$ exists such that $k''$ is linearly independent of $k'$ and $\langle\partial h(p),k''\rangle=0$ holds. Recall the process of KAM iteration, the main part of the resonant term is obtained by averaging the perturbation over a circle determined by these two resonant relations. It takes the form
$$
Z=Z_{k'}(p,\langle k', q\rangle)+Z_{k',k''}(p,\langle k', q\rangle,\langle k'',q\rangle)
$$
where
$$
Z_{k'}=\sum_{j\in\mathbb{Z}\backslash\{0\}}P_{jk'}(p)e^{j\langle k',q\rangle i}, \qquad
Z_{k',k''}=\sum_{(j,l)\in\mathbb{Z}^2, l\neq 0}P_{jk'+lk''}(p)e^{(j\langle k',q\rangle+l\langle k'',q\rangle)i}.
$$
Since $P$ is $C^r$-function, the coefficient $P_{jk'+lk''}$ is bounded by
$$
|P_{jk'+lk''}|\le 8\pi^3\|P\|_{C^r}\|jk'+lk''\|^{-r},
$$
which induces the estimation
\begin{equation}\label{criterioneq1}
\|Z_{k',k''}\|_2\le d\|P\|_{C^r}\|k''\|^{-r+2}
\end{equation}
where $d=d(k')$ depends on $k'$. Just like the procedure we did in the second section, after the rescaling and linear coordinate transformation we get the main part of the system
$$
G_0=\frac 12\langle Ay,y\rangle-V_{k'}(x_2)-V_{k',k''}(x_1,x_2).
$$
Assume $V_{k'}$ has a non-degenerate minimal point at $x_2^*$, i.e. $\ddot V_{k'}(x_2^*)=\lambda>0$, the system $\langle Ay,y\rangle-V_{k'}(\langle k',q\rangle)$ possesses a NHIC
$$
\Pi_{k',k''}^0=\{y=\xi y_0,\xi\in\mathbb{R},x_2=x_2^*,x_1\in\mathbb{T}\}
$$
where $y_0$ solves the equation $(1,0)^t=Ay_0$. Applying the normally hyperbolic invariant manifold theorem, one obtains from the estimate (\ref{criterioneq1}) that some positive number $d_1=d_1(\lambda)>0$ exists such that $\Phi_{G_0}^t$ also admits a normally hyperbolic and invariant cylinder $\Pi_{k',k''}$ close to $\Pi_{k',k''}^0$ provided
\begin{equation}\label{criterioneq2}
\|k''\|^{r-2}\ge\frac {d(k')}{d_1(\lambda)}\|P\|_{C^r}.
\end{equation}
It is a criterion to see whether the double resonance is thought as weak resonance and can be treated in the way for {\it a priori} unstable system.

However, we obtain the potential $V$ by fixing $y=y''$. So, the non-degeneracy of the minimal point depends on the position of double resonant point on the resonant path $\Gamma'$, i.e. the number $\lambda$ depends on the $y\in\Gamma'$. Because the set of double resonant points is dense along the resonant path, it appears necessary to ask whether it holds simultaneously for all $p\in\Gamma'$ that the minimal point of $Z_{k'}(p,x)$ is non-degenerate when it is treated as a function of $x=\langle k', q\rangle$. Fortunately, we have the following result \cite{CZ2}
\begin{theo}\label{singlecylinder}
Assume $M$ is a closed manifold with finite dimensions, $F_{\zeta}\in C^r(M,\mathbb{R})$ with $r\ge 4$ for each $\zeta\in[\zeta_0,\zeta_1]$ and $F_{\zeta}$ is Lipschitz in the parameter $\zeta$. Then, there exists an open-dense set $\mathfrak{V}\subset C^r(M,\mathbb{R})$ so that for each $V\in\mathfrak{V}$, it holds simultaneously for all $\zeta\in[\zeta_0,\zeta_1]$ that the minimal of $F_{\zeta}+V$ is non-degenerate. In fact, given $V\in\mathfrak{V}$ there are finitely many $\zeta_i\in[\zeta_0,\zeta_1]$ such that $F_{\zeta}+V$ has only one global minimal $($maximal$)$ point for $\zeta\ne\zeta_i$ and has two global minimal $($maximal$)$ point if $\zeta=\zeta_i$.
\end{theo}
So, once one has a generic single resonant term $Z_{k'}$, the non-degeneracy $\lambda$ is lower bounded from zero for all double resonant points. There are finitely many $k''\in\mathbb{Z}^3$ which do not satisfy the condition (\ref{criterioneq2}), thus need to be treated as strong double resonance. Obviously, the number of such points is independent of $\epsilon$.

It follows from Theorem \ref{singlecylinder} that there are finitely many point $p=p'_j\in\Gamma'$ where the single resonant term $Z_{k'}$ has two global minimal points when it is treated as the function $\langle k',q\rangle$. It is clearly generic that the condition ({\bf H3}) holds for $Z_{k'}$. It implies that as one move $p$ along $\Gamma'$, the Mather set varies along one cylinder and jump to another cylinder when it crosses the point $p'_j$ which is called bifurcation point. It is also generic that none of these bifurcation points is strong double resonant point.
\end{proof}

\noindent{\bf Remark}. Given a resonant path determined by a class $g\in H_1(\mathbb{T}^2,\mathbb{Z})$, we obtain a channel $\mathbb{C}_g=\cup_{\lambda}\mathscr{L}_{\beta_{G_{\epsilon,0}}}(\lambda g)\subset H^1(\mathbb{T}^2, \mathbb{R})$. By the result we get in this paper, this channel has certain width except the place very close to the disk which corresponds to strong double resonance $\{\mathbb{F}_i\}$. For each $c\in\mathrm{int}\mathbb{C}_g$ with $d(c,\mathbb{F}_i)>O(\sqrt{\epsilon}^{1+d})$, the Aubry set is located in certain NHIC \cite{Mas}. By using the method of \cite{CY1,CY2,LC}, this Aubry set can be connected to other Aubry set nearby which is also on the cylinder. Sometimes, the local connecting orbit has to be the type looks like heterclinic orbit (Arnold's mechanism), sometimes the local connecting orbit has to be constructed by using cohomology equivalence. However, due to the work of \cite{Z1}, one can always connects such Aubry set to another one via Arnold's mechanism. Certain H\"older modulus continuity of weak KAM solutions is established in \cite{Z1} for the whole cylinder, not only restricted on the set of invariant circles.

Therefore, we have found transition chain along the resonant path with only finitely gaps around the strong double resonant points. These gaps have very small size like $O(\sqrt{\epsilon}^{1+d})$ with $d>0$. In another paper \cite{C15}, we show how to build up a transition chain of cohomology equivalence to cross the double resonance. We note that it was announced by Mather \cite{Mat} more than ten years ago.

\noindent{\bf Acknowledgement} The main content of this paper comes from the previous preprint \cite{C13} which seems quite long. The author splits it into several parts with modifications so that they are easier to read. This paper is one of them.

This work is supported by National Basic Research Program of China (973 Program, 2013CB834100), NNSF of China (Grant 11171146) and a program PAPD of Jiangsu Province, China.

\end{document}